\documentclass[11pt,a4paper,oneside,reqno]{amsart}

\usepackage{graphicx}

\newcommand{\dd}{\mathrm{d}}

\newcommand{\RR}{\mathbb{R}}

\newcommand{\NN}{\mathbb{N}}
\newcommand{\cO}{\mathcal{O}}

\newcommand{\cN}{\mathcal{N}}

\newtheorem{theorem}{Theorem}
\newtheorem{prop}[theorem]{Proposition}
\newtheorem{lemma}[theorem]{Lemma}
\newtheorem{corol}[theorem]{Corollary}

\theoremstyle{definition}

\DeclareMathOperator{\spec}{spec}

\DeclareMathOperator{\supp}{supp}
\DeclareMathOperator{\diag}{diag}
\DeclareMathOperator{\Span}{span}

\begin{document}

\title[]{Eigenvalue inequalities and absence of threshold resonances
for waveguide junctions}

\author{Konstantin Pankrashkin}

\address{Laboratoire de Math\'ematiques d'Orsay, Univ.~Paris-Sud, CNRS, Universit\'e Paris-Saclay, 91405 Orsay, France}

\email{konstantin.pankrashkin@math.u-psud.fr}
\urladdr{http://www.math.u-psud.fr/~pankrash/}

\begin{abstract}
Let $\Lambda\subset \mathbb{R}^d$ be a domain consisting
of several cylinders attached to a bounded center.
One says that $\Lambda$ admits a threshold
resonance if there exists a non-trivial bounded function $u$ solving
 $-\Delta u=\nu u$ in $\Lambda$ and vanishing at the boundary,
where $\nu$ is the bottom of the essential spectrum of the Dirichlet Laplacian in $\Lambda$.
We give a sufficient condition for the absence of threshold resonances
in terms of the Laplacian eigenvalues on the center.
The proof is elementary and is based on the min-max principle.
Some two- and three-dimensional examples and applications to the study of Laplacians on thin networks
are discussed.
\end{abstract}

\maketitle

\section{Introduction}

Let $\Lambda\subset \RR^d$, $d\ge 2$, be a connected Lipschitz domain 
which can be represented as a family of several half-infinite cylinders
attached to a bounded domain. More precisely, we assume that there exist
bounded connected Lipschitz domains $\omega_j\subset \RR^{d-1}$, called cross-sections,
and $n$ non-intersecting half-infinite cylinders $B_1,\dots, B_n \subset \Lambda$,
isometric respectively to $\RR_+\times \omega_j$, $\RR_+:=(0,+\infty)$,
such that $\Lambda$ coincides with the union $B_1\mathop{\cup}\dots \mathop{\cup} B_n$
outside a compact set, see Figure~\ref{fig1}(a). The cylinders $B_j$ will be called \emph{branches},
  the connected bounded domain
$C:=\Lambda\setminus \overline{\mathstrut B_1\mathop{\cup}\dots \mathop{\cup} B_n}$
will be called \emph{center}, and we assume that the boundary of $C$
is Lipschitz too. We call such a domain $\Lambda$ a \emph{star waveguide}.
Remark that the choice of a center in not unique: any center can be enlarged
by including finite pieces of the branches, see Figure~\ref{fig1}(b).

In the present work, we would like to establish some elementary
conditions guaranteeing
the non-existence of non-trivial bounded solutions to
\begin{equation}
     \label{eq-bd}
-\Delta u=\nu u \text{ in $\Lambda$ }, \quad
u=0 \text{ at $\partial\Lambda$},
\end{equation}
where $\nu$ is the bottom of the essential spectrum of the Dirichlet
Laplacian $-\Delta^\Lambda_D$ acting in $L^2(\Lambda)$. It is standard to see
that $\nu=\min \nu_j$, where $\nu_j$ is the lowest Dirichlet eigenvalue
of the cross-section $\omega_j$, and the spectrum of $-\Delta^\Lambda_D$
consists of the semi-axis $[\nu,\infty)$ and of a finite family
of discrete eigenvalues $\lambda_j(-\Delta^\Lambda_D)$,
$j\in\big\{1,\dots,N(\Lambda)\big\}$, while the case $N(\Lambda)=0$ (no discrete eigenvalues)
is possible. As shown e.g. in \cite[Theorem 4]{mv}, a non-trivial bounded solution
of \eqref{eq-bd} exists iff the resolvent $z\mapsto (-\Delta^\Lambda_D-z)^{-1}$
has a pole at $z=\nu$, and in that case we say that $\Lambda$
admits a threshold resonance.

\begin{figure}
\centering
\centering
\begin{tabular}{cc}
\begin{minipage}[c]{65mm}
\begin{center}
\includegraphics[height=20mm]{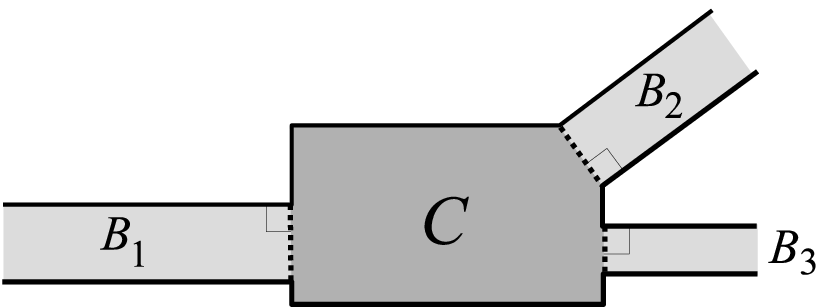}
\end{center}
\end{minipage}
&
\begin{minipage}[c]{65mm}
\begin{center}
\includegraphics[height=20mm]{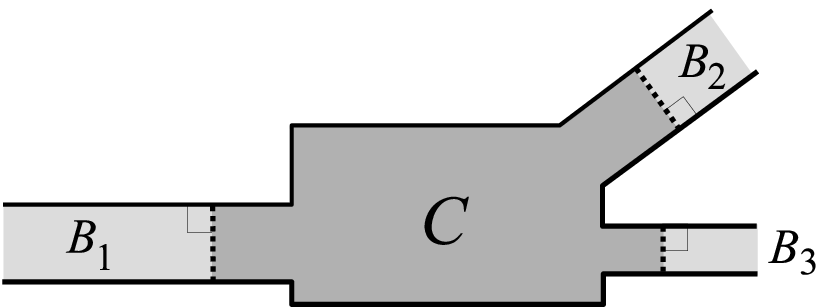}
\end{center}
\end{minipage}\\
(a) & (b)
\end{tabular}
\caption{ (a) An example of a star waveguide $\Lambda$ with three branches and a dark-shaded center.
(b) An alternative choice of a center. \label{fig1}}
\end{figure}

The study of  threshold resonances is motivated, in particular,
by the analysis of Dirichlet Laplacians in systems of thin tubes
collapsing onto a graph.
Namely, for a small $\varepsilon>0$,
consider a domain $\Omega_\varepsilon\subset \RR^d$ composed of
finitely many cylinders (''edges'')  $B_{j,\varepsilon}$ isometric to
$I_j\times (\varepsilon \omega)$ with $I_j:=(0,\ell_j)$, $\ell_j\in\RR_+$, $j\in\{1,\dots,J\}$,
connected to a ''network'' through some bounded Lipschitz
domains (''vertices'') $C_{k,\varepsilon}$, see Figure~\ref{fig-net}(a). (The case of nonidentical cross-sections
is also possible but the formulations become more complicated.)
We assume that the vertices $C_{k,\varepsilon}$
are isometric to $\varepsilon C_{k}$, where the domains
$C_k$ are $\varepsilon$-independent, $k\in\{1,\dots,K\}$,
and that the pieces are glued together in such a way that if one considers
a vertex $C_{k,\varepsilon}$
and extends the attached edges to infinity,
then one obtains a domain isometric to $\varepsilon \Lambda_k$,
where $\Lambda_k$ is an $\varepsilon$-independent
star waveguide whose center is $C_k$.

\begin{figure}
\centering
\begin{tabular}{cc}
\begin{minipage}[c]{65mm}
\begin{center}
\includegraphics[height=25mm]{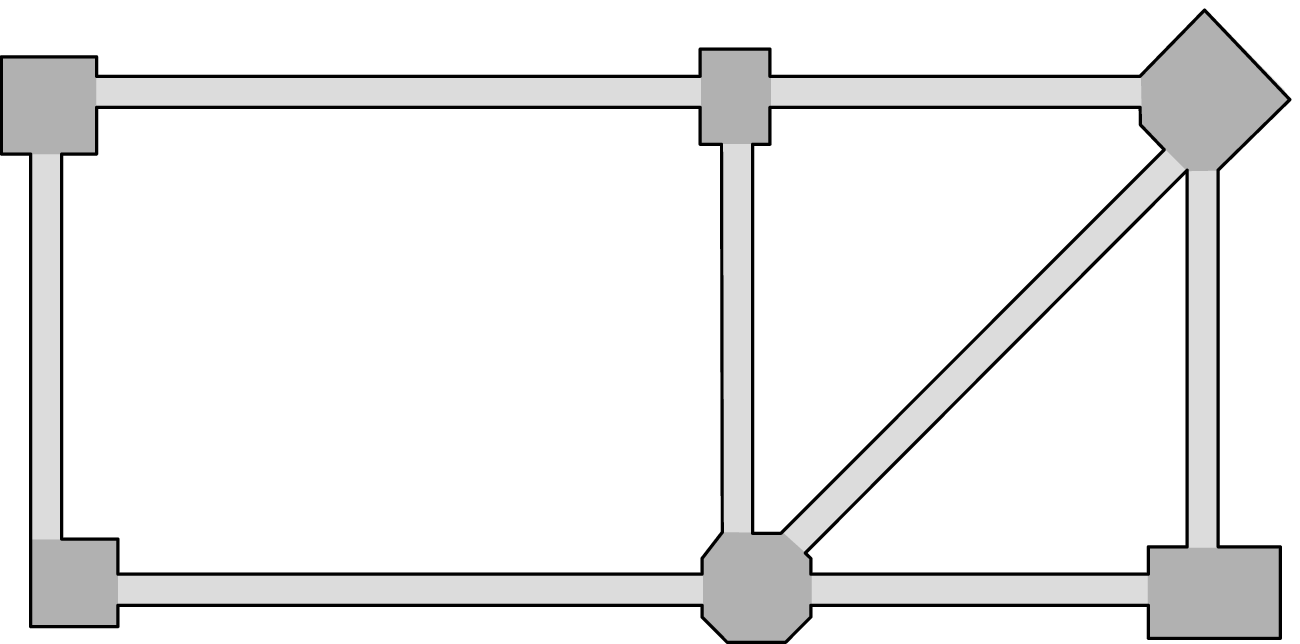}
\end{center}
\end{minipage}
&
\begin{minipage}[c]{65mm}
\begin{center}
\includegraphics[height=22mm]{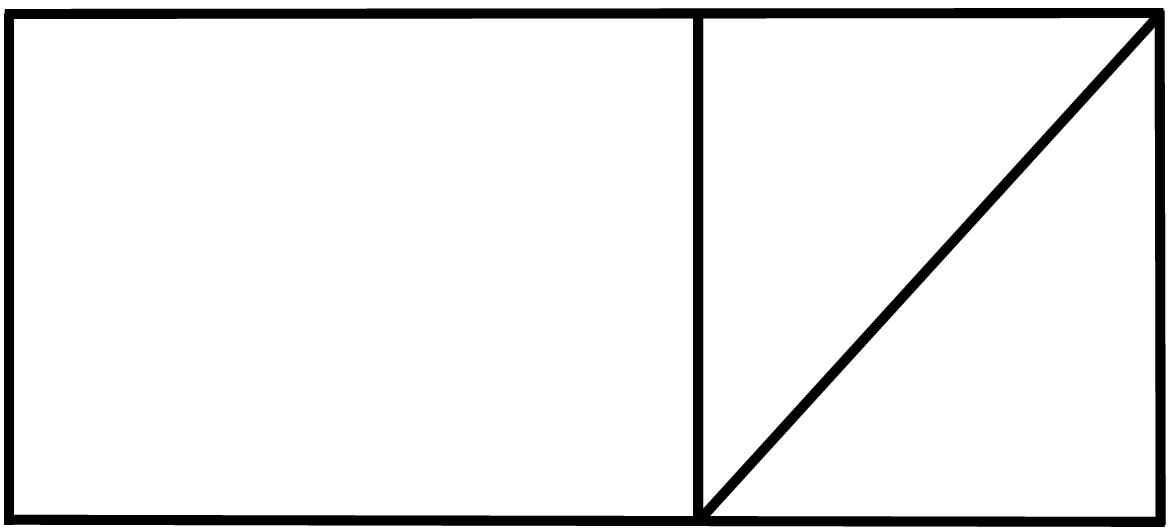}
\end{center}
\end{minipage}\\
(a) & (b)
\end{tabular}
\caption{ (a) An example of a domain $\Omega_\varepsilon$. The vertex parts are dark-shaded.
(b) The associated one-dimensional skeleton $X$. \label{fig-net}}
\end{figure}

Denote by $-\Delta^{\Omega_\varepsilon}_D$ the Dirichlet Laplacian
in $\Omega_\varepsilon$. In various applications, one is interested
in the asymptotics of its eigenvalues $\lambda_m(-\Delta^{\Omega_\varepsilon}_D)$
as $\varepsilon$ tends to $0$, see e.g. the monographs~\cite{ekbook,post-book}
and the reviews~\cite{grieser2,kuch}.
As the domain $\Omega_\varepsilon$ collapses onto it one-dimensional skeleton
$X$ composed from the intervals $I_j$ coupled at the vertices, see
Figure~\ref{fig-net}(b), one may expect that the eigenvalue asymptotics
might be determined
by some effective operator acting on the functions defined on $X$.
The results obtained by several authors, see e.g. \cite{grieser,mv}, can be 
informally summarized as follows. Consider the star waveguides $\Lambda_k$ associated
to each vertex as described above,
the associated Dirichlet Laplacians  $-\Delta^{\Lambda_k}_D$ and their discrete
eigenvalues $\lambda_j(-\Delta^\Lambda_D)$, $j\in \{1,\dots,N(\Lambda_k)\}$, $k\in\{1,\dots,K\}$,
then the bottom of the essential spectrum is exactly the first Dirichlet eigenvalue
$\nu$ of the cross-section $\omega$, and the following holds
as $\varepsilon$ tends to $0$: there exists $N\ge N(\Lambda_1)+\dots+N(\Lambda_K)$ such that
\begin{itemize}
\item for $m\in \{1,\dots,N\}$ there holds
\[
\lambda_m(-\Delta^{\Omega_\varepsilon}_D)=\dfrac{a_m}{\varepsilon^2}+\cO(e^{-c_m/\varepsilon})
\text{ with $a_m\in (0,\nu]$ and $c_m>0$},
\]
\item for any  $m\ge 1$ there holds
\[
\lambda_{N+m}(-\Delta^{\Omega_\varepsilon}_D)=
\dfrac{\nu}{\varepsilon^2} + \mu_m+\cO(\varepsilon),
\]
where $\mu_m$ are the eigenvalues of the self-adjoint operator
$L$ in $\bigoplus_{j=1}^J L^2(0,\ell_j)$ acting as $(f_j)\mapsto (-f''_j)$
with suitable self-adjoint boundary conditions determined by the scattering matrices
of $-\Delta^{\Lambda_k}_D$ at the energy~$\nu$.
\end{itemize}
The operator $L$, which is the so-called quantum graph laplacian on $X$~\cite{bk,post-book},
represents the sought ''effective operator'', and the associated boundary conditions
describe the way how the branches of the network interact through the vertices
in the limit $\varepsilon\to 0$.
An exact formulation, including 
the case of non-identical cross-sections, is presented in~\cite[Theorems 2 and 3]{grieser},
but is is quite complicated and needs a number of precise definitions, and finding the
boundary conditions for $L$ is a non-trivial transcendental problem,
but the whole construction admits an important particular case giving the following
simple result, see \cite[Section~8]{grieser} and \cite[Theorem 7]{mv}:
\begin{prop}\label{prop1}
Assume that none of $\Lambda_k$ admits a threshold resonance, then:
\begin{itemize}
\item Denote $N:= N(\Lambda_1)+\dots+N(\Lambda_K)$ and let $a_1,\dots,a_N$ be
the eigenvalues $\lambda_j(-\Delta^{\Lambda_k}_D)$, $j\in\{1,\dots,N(\Lambda_k)\}$, $k\in\{1,\dots,K\}$,
enumerated in the non-decreasing order,
then for $m\in\{1,\dots,N\}$ there holds, with some $c_m>0$,
\[
\lambda_m(-\Delta^{\Omega_\varepsilon}_D)=\dfrac{a_m}{\varepsilon^2}+\cO(e^{-c_m/\varepsilon})
\quad \text{as $\varepsilon$ tends to $0$},
\]
\item for any $m\ge 1$ there holds
\[
\lambda_{N+m}(-\Delta^{\Omega_\varepsilon}_D)=
\dfrac{\nu}{\varepsilon^2} + \mu_m+\cO(\varepsilon)
\quad \text{as $\varepsilon$ tends to $0$},
\]
where $\mu_m$ is the $m$th eigenvalue of $D_1\mathop{\oplus}\dots \oplus D_J$,
with $D_j$ being the Dirichlet Laplacian on $(0,\ell_j)$.
\end{itemize}
\end{prop}
In other words, in the absence of threshold resonances
the effective operator $L$ is decoupled. Numerous papers claimed that the assumptions of Proposition~\ref{prop1}
are generically satisfied, i.e. are true for ``almost any'' star waveguide,
which is supported by various analytical arguments, see e.g. \cite{dell,grieser,grieser2,mv}.
Nevertheless, there are only few results guaranteeing the non-existence of a threshold resonance
for an explicitly given configuration. In fact, the only explicitly formulated condition
we are aware of is the one appearing e.g. in \cite[Theorem 25]{grieser},
which applies to the above star waveguide $\Lambda$:
\begin{prop}\label{prop2}
Let $C$ be a center of $\Lambda$. Denote by $-\Delta^C_{DN}$ the Laplacian
in $L^2(C)$ with the Dirichlet boundary condition at $\partial C \mathop{\cap}\partial\Lambda$
and with the Neumann boundary condition at the remaining part of the boundary (e.g.
on the dash lines in Figure~\ref{fig1}).
If one has the strict inequality
\begin{equation}
  \label{eq-lnu}
\lambda_1(-\Delta^C_{DN})>\nu,
\end{equation}
then $\Lambda$ has no threshold resonance.
\end{prop}
Recall that, by the min-max principle, for any $j\in \{1,\dots,N(\Lambda)\}$ there holds
\begin{equation}
      \label{eq-lnu2}
\lambda_j(-\Delta^\Lambda_D)\ge \lambda_j(-\Delta^C_{DN}).
\end{equation}
Therefore, in the situation of Proposition~\ref{prop2}
the operator $-\Delta^\Lambda_D$ has no discrete
eigenvalues, i.e. $N(\Lambda)=0$, and its spectrum is $[\nu,+\infty)$.
In particular, if one has a network $\Omega_\varepsilon$ of the above type
and such that the star waveguide associated with each vertex
satisfies the assumptions of Proposition~\ref{prop2}, then the result of Proposition~\ref{prop1}
takes a simpler form, as one simply has $N=0$.
One should remark that this particular case of Proposition~\ref{prop1}
was initially proved in \cite{post-dir} in a direct way, without
explicit link to the threshold resonances. The condition \eqref{eq-lnu}
is usually interpreted as the smallness of the center of the star waveguide
with respect to the thickness of the branches. This situation
is quite special, and it is generally expected
that deformed waveguides of constant width
have discrete eigenvalues~\cite{bulla,es,gold,bind,naz12}.

Recently, some specific star waveguide configurations
in two and three dimensions were studied in \cite{bakh,naz-t,naz-hex},
and the absence of threshold resonances was shown.
One should remark that, in all the cases considered, the
condition \eqref{eq-lnu} is not satisfied,
and a non-empty discrete spectrum is present. The aim
of the present paper is to state explicitly the main
condition used in the constructions of~\cite{bakh,naz-t,naz-hex}
and then to show how it can be applied to the analysis
of more general geometric configurations. 
Our main contribution is as follows:
\begin{theorem}\label{thm1}
Let $C$ be a center of $\Lambda$ and $-\Delta^C_{DN}$ be as in Proposition~\ref{prop2}.
If
\begin{equation}
   \label{eq-lnun}
 \lambda_{N(\Lambda)+1}(-\Delta^C_{DN})>\nu,
\end{equation}
then $\Lambda$ has no threshold resonance.
\end{theorem}
As noted above, Proposition~\ref{prop2} is a special case of Theorem~\ref{thm1}
with $N(\Lambda)=0$.
For further references, let us state explicitly another obvious but important corollary
corresponding to $N(\Lambda)=1$, which is essentially the condition used in \cite{bakh,naz-t,naz-hex}:
\begin{corol}\label{cor1}
If the discrete spectrum $-\Delta^\Lambda_D$ is non-empty
and for some center $C$ one has
 $\lambda_{2}(-\Delta^C_{DN})>\nu$, then $-\Delta^\Lambda_D$ has a single discrete eigenvalue
and no threshold resonance.
\end{corol}

The proof of Theorem~\ref{thm1} is given in the following section,
and it is quite elementary. We show first,
using an explicit construction of test functions,
that the presence of a threshold resonance gives
rise to additional eigenvalues if one perturbs the Dirichlet
Laplacian in $\Lambda$ by a negative potential.
Then we show that such a behavior contradicts the assumption~\eqref{eq-lnun}.
In fact, a similar scheme was used in \cite{naz-t,naz-hex}
but with a different type of perturbation. Our choice of
a potential perturbation allows for a more straightforward
use of the min-max principle, and the resulting proof
appears to be less technical.

In Section~\ref{sexam} we present several explicit examples in two
and three dimensions in which the assumptions of Theorem~\ref{thm1}
can be verified. Remark that the example given in subsection~\ref{ssmany}
is not covered by Corollary~\ref{cor1}.

We remark at last that the Dirichlet boundary condition at the boundary
of $\Lambda$ is only taken as an example,
it can be replaced by some others such as Robin or mixed ones.
Note that for the Neumann boundary condition one always has $\nu=0$, and
there is a threshold resonance corresponding to the constant solutions of $-\Delta u=0$.
In this case one always has $N(\Lambda)=0$,
the operator $-\Delta^C_{DN}$ should be replaced by the Neumann Laplacian
on $C$, whose first eigenvalue is $0=\nu$, and Eq.~\eqref{eq-lnun}
is never satisfied.

\section{Proof of Theorem~\ref{thm1}}\label{s1}

The proof is by assuming the opposite. We first show  (Lemma~\ref{lem5})
that if $\Lambda$ has a threshold
resonance, then any perturbation of some class produces an additional eigenvalue,
which is done by constructing a  family of suitable test functions.
On the other hand, in Lemma~\ref{lem8} we show that under the assumption
\eqref{eq-lnun} one can construct a perturbation
of this class producing no new eigenvalues, which gives the result.

Recall that for a set $A$ we denote by $1_A$ its indicator function, which is defined
by $1_A(x)=1$ for $x\in A$ and $1_A(x)=0$ otherwise.

\subsection{Perturbations producing additional eigenvalues}\label{ss21}

This subsection is devoted to the proof of the following assertion.
 
\begin{lemma}\label{lem5}
Assume that $\Lambda$ has a threshold resonance. Let $\Omega\subset\Lambda$
be a non-empty bounded open set, then for any $\gamma>0$
the perturbed operator $-\Delta^\Lambda_D-\gamma 1_\Omega$ has at least
$N(\Lambda)+1$ eigenvalues in $(-\infty,\nu)$.
\end{lemma}

The perturbation is compactly supported and does not change the essential spectrum,
and by the min-max principle it is sufficient to find a $\big(N(\Lambda)+1\big)$-dimensional
subspace $V\subset H^1_0(\Lambda)$ with
\begin{equation}
      \label{eq-mm1}
\sup_{v\in V\setminus\{0\}}
\dfrac{\|\nabla v\|^2_{L^2(\Lambda)} - \gamma \|v\|^2_{L^2(\Omega)}}{\|v\|^2_{L^2(\Lambda)}}<\nu.
\end{equation}
By assumption, there exists a non-zero bounded solution $u_0$ of \eqref{eq-bd}. Denote for brevity $N:=N(\Lambda)$ and $\lambda_j:=\lambda_j(-\Delta^\Lambda_D)$,
$j\in\{1,\dots,N\}$, and choose an associated orthonormal family
of eigenfunctions $u_j$ of $-\Delta^\Lambda_D$,
\begin{equation}
   \label{eq-uu}
\langle u_j,u_k\rangle_{L^2(\Lambda)}=\delta_{jk},
\quad
-\Delta u_j=\lambda_j u_j, \quad j\in\{1,\dots, N\}.
\end{equation}
Note that the functions $u_0, \dots,u_N$ are smooth in $\Lambda$
due to the elliptic regularity. Let us emphasize another simple
property:
\begin{lemma}\label{lem1}
The functions $u_0,\dots,u_N$ are linearly independent
on any non-empty open subset of~$\Lambda$.
\end{lemma}
\begin{proof}
Assume the opposite, i.e. that there exist a non-empty open subset $U\subset\Lambda$
and $\xi=(\xi_0,\dots,\xi_N)\in\RR^{N+1}\setminus\{0\}$
such that
\begin{equation}
     \label{eq-xu}
\sum\nolimits_{j=0}^N \xi_j u_j=0 \text{ in } U.
\end{equation}
Denote $\lambda_0:=\nu$ and   $\Sigma:=\{\lambda_0,\dots, \lambda_N\big\}$,
pick any $\lambda\in \Sigma$ and apply successively the differential
expressions $(-\Delta-\mu)$ with all $\mu\in\Sigma\setminus\{\lambda\}$
to Eq.~\eqref{eq-xu}. We arrive at
\[
\Big(\prod_{\mu\in \Sigma\setminus\{\lambda\}}(\lambda-\mu) \Big)v_\lambda=0 \text{ in } U,
\quad v_\lambda:=\sum\nolimits_{j:\, \lambda_j=\lambda} \xi_j u_j,
\]
and the function $v_\lambda$ must vanish in $U$.
On the other hand, it satisfies $-\Delta v_\lambda=\lambda v_\lambda$ in $\Lambda$,
hence, $v_\lambda\equiv 0$ in $\Lambda$ due to the unique continuation principle.
In particular, for $\lambda=\lambda_0=\nu$ we obtain $\xi_0 u_0=0$ in $\Lambda$, and
$\xi_0=0$ as $u_0$ is not identically zero.
For $\lambda=\lambda_k$ with $k\ne 0$ we arrive at
\[
\sum\nolimits_{j:\,\lambda_j=\lambda_k} \xi_j u_j=0 \text{ in } \Lambda,
\]
implying $\xi_j=0$ for all $j$ with $\lambda_j=\lambda_k$, as the family $(u_1,\dots, u_n)$
is orthonormal.
Therefore, $\xi_j=0$ for all $j\in\{0,\dots, N\}$, which is in contradiction
with $\xi\ne 0$.
\end{proof}

Let us pick a $C^\infty$ cut-off function
$ \chi:\RR\to[0,1]$ with $\chi(r)=1$  for $r\le 1$
and $\chi(r)=0$ for $r\ge 2$, and define $\varphi:\Lambda\to \RR$ by
$\varphi(x)= \chi\big(|x|/R\big)$ with some $R>R_0$,
where $R_0$ is sufficiently large to have
$\varphi=1$ on $\Omega$.
Now set
\[
v_0:=\varphi u_0,
\quad
v_j=u_j, \quad j\in\{1,\dots,N\}.
\]
\begin{lemma}\label{lem2}
The functions $v_0,\dots,v_N$ are linearly independent in $L^2(\Lambda)$
for any $R>R_0$.
\end{lemma}
\begin{proof}
By construction, the functions are in $L^2(\Lambda)$.
Furthermore, one has $v_j=u_j$ in $\Omega$ for $R>R_0$,
and the result follows from Lemma~\ref{lem1}.
\end{proof}
Now we are going to  show the inequality \eqref{eq-mm1} for $V:=\Span(v_0,\dots,v_N)$
with a large $R$. It is sufficient to show that
\begin{multline}
     \label{eq-mm2}
\sup_{\xi\in \RR^{n+1},\,|\xi|=1}
\bigg(
\Big\|\sum\nolimits_{j=0}^N \xi_j \nabla v_j \Big\|^2_{L^2(\Lambda)} -\nu \Big\|\sum\nolimits_{j=0}^N \xi_j  v_j\Big\|^2_{L^2(\Lambda)}- \gamma \Big\|\sum\nolimits_{j=0}^N \xi_j v_j\Big\|^2_{L^2(\Omega)}
\bigg)\\
\equiv
\sup_{\xi\in \RR^{n+1},\,|\xi|=1} \big\langle \xi, (A-\gamma B) \xi\big\rangle_{\RR^{n+1}}<0,
\end{multline}
with $A=(a_{jk})$, $B=(b_{jk})$,
\[
a_{jk}:=\int_{\Lambda} \nabla v_j\cdot \nabla v_k\dd x
- \nu \int_{\Lambda}  v_j v_k\dd x,
\quad
b_{jk}:=\int_{\Omega}  v_j v_k\dd x, \quad
j,k\in\{0,\dots,N\}.
\]
More precisely, the coefficients of $A$ are
\begin{equation}
    \label{eq-coefa}
\begin{aligned}
a_{00}&=\int_{\Lambda} |\nabla(\varphi u_0)\big|^2\,\dd x - \nu\int_{\Lambda} (\varphi u_0)^2\, \dd x,\\
a_{j0}&=a_{0j}=\int_{\Lambda} \nabla(\varphi u_0)\cdot \nabla u_j\,\dd x - \nu\int_{\Lambda} \varphi u_0 u_j\, \dd x,
\quad j\in\{1,\dots,N\big\},\\
a_{jk}&= (\lambda_j-\nu)\delta_{jk}, \quad j,k\in\{1,\dots,N\big\}. \nonumber
\end{aligned}
\end{equation}
To estimate $a_{00}$ we remark that
\[
\int_{\Lambda} |\nabla(\varphi u_0)\big|^2\,\dd x 
= \int_{\Lambda} |\nabla \varphi|^2 u_0^2\,\dd x
+ \int_{\Lambda} \varphi^2 |\nabla u_0|^2\,\dd x
+2 \int_{\Lambda} \varphi u_0 \nabla\varphi\cdot\nabla u_0\, \dd x,
\]
and an integration by parts gives
\begin{multline*}
2 \int_{\Lambda} \varphi u_0 \nabla\varphi\cdot\nabla u_0\, \dd x=
\int_{\Lambda} \nabla( \varphi^2) \cdot \big( u_0\nabla u_0)\, \dd x=-\int_{\Lambda} \varphi^2 \nabla \cdot ( u_0\nabla u_0)\, \dd x\\
=-\int_{\Lambda} \varphi^2 |\nabla u_0|^2 \dd x
+\int_{\Lambda} \varphi^2 (-\Delta u_0) u_0\, \dd x
=-\int_{\Lambda} \varphi^2 |\nabla u_0|^2 \dd x
+\nu \int_{\Lambda} \varphi^2 u_0^2\, \dd x,
\end{multline*}
resulting in
\[
a_{00}=\int_{\Lambda} |\nabla \varphi|^2 u_0^2\,\dd x.
\]
For large $R$ there holds $\|\nabla\varphi\|_\infty\le R^{-1}\|\chi'\|_\infty=\cO(R^{-1})$,
and the volume of $\Lambda\mathop{\cap}\supp \nabla\varphi$ is $\cO(R)$. Hence, due to the boundedness
of $u_0$ there holds $a_{00}=\cO(R^{-1})$ as $R\to+\infty$.

To estimate $a_{j0}$ with $j\ne0$ we remark first that
\begin{align*}
\int_{\Lambda} \nabla(\varphi u_0)\cdot \nabla u_j\, \dd x
&=-\int_{\Lambda} \Delta(\varphi u_0) u_j\, \dd x\\
&=\int_{\Lambda} (-\Delta\varphi) u_0 u_j\, \dd x
-2\int (u_j\nabla \varphi )\cdot \nabla u_0\, \dd x
+\int_{\Lambda} \varphi (-\Delta u_0) u_j\, \dd x\\
&=\int_{\Lambda} (-\Delta\varphi) u_0 u_j\, \dd x
+2\int_{\Lambda} u_0\nabla\cdot (u_j\nabla \varphi )\, \dd x
+\nu \int_{\Lambda} \varphi  u_0  u_j\, \dd x\\
&= \int_{\Lambda} (\Delta \varphi)u_0 u_j\, \dd x
+2 \int_{\Lambda} u_0\nabla u_j\cdot \nabla \varphi \, \dd x
+\nu \int_{\Lambda} \varphi  u_0  u_j\, \dd x,
\end{align*}
hence,
\[
a_{j0}=a_{0j}= \int_{\Lambda} (\Delta \varphi)u_0 u_j\, \dd x
+2 \int_{\Lambda} u_0\nabla u_j\cdot \nabla \varphi \, \dd x.
\]
We estimate, using the Cauchy-Schwarz inequality,
\begin{multline*}
\bigg|
\int_{\Lambda} u_0\nabla u_j\cdot \nabla \varphi \, \dd x
\bigg|
\le
\int_{\Lambda} |\nabla u_j| \cdot |u_0\nabla \varphi| \, \dd x\\
\le
\sqrt{\int_{\Lambda} |\nabla u_j|^2\, \dd x}\cdot
\sqrt{\int_{\Lambda} |\nabla \varphi|^2 u_0^2\, \dd x}
=\cO\Big(\dfrac{1}{\sqrt{R}}\Big), \quad R\to+\infty.
\end{multline*}
Due to
\[
\Delta\varphi(x)= \dfrac{1}{R^2}\,\chi''\Big(\dfrac{|x|}{R}\Big)+\dfrac{d-1}{|x|}\,\dfrac{1}{R}\,\chi'\Big(\dfrac{|x|}{R}\Big)
\]
one has $\|\Delta \varphi\|_\infty=\cO(R^{-1})$ for large $R$. At the same time, the volume of $\Lambda\mathop{\cap}\supp(\Delta \varphi)$ is $\cO(R)$ and $u_0$ is bounded, therefore,
\[
\bigg|\int_{\Lambda} (\Delta \varphi)u_0 u_j\, \dd x\bigg|
\le \sqrt{\int_{\Lambda} (\Delta \varphi)^2u_0^2\, \dd x} \cdot
\sqrt{\int_{\Lambda} u_j^2\, \dd x}=\cO\Big( \dfrac{1}{\sqrt R}\Big),
\]
hence,
$a_{j0}=a_{0j}=\cO(R^{-\frac 12})$ as $R\to+\infty$ for $j\in\{1,\dots, N\big\}$, and,
\[
A=\diag(
0, \lambda_1-\nu, \dots, \lambda_N-\nu
)
+\cO(R^{-\frac 12}), \quad R\to+\infty.
\]
In particular, for a suitable $a>0$ there holds
\begin{equation}
    \label{eq-aaa}
\sup_{\xi\in \RR^{n+1},\, |\xi|=1} \langle \xi, A\xi\rangle_{\RR^{n+1}} \le a R^{-\frac 12} \text{ for } R\to+\infty.
\end{equation}
To estimate $B$ we remark that for  $R>R_0$ one has $v_j=u_j$ in $\Omega$, and
\[
b_{jk}=\int_{\Omega} u_j\, u_k\dd x, \quad j,k\in\{0,\dots,N\}.
\]
Hence, due to the compactness of the unit ball of $\RR^{n+1}$
and to Lemma~\ref{lem1} there holds
\begin{equation}
       \label{eq-bbb}
\inf_{\xi\in \RR^{n+1},\, |\xi|=1} \langle \xi,B\xi\rangle_{\RR^{n+1}}
=
\inf_{\xi\in \RR^{n+1},\, |\xi|=1} \Big\|
\sum\nolimits_{j=0}^N \xi_j u_j \Big\|_{L^2(\Omega)}^2=:b>0.
\end{equation}
The combination of \eqref{eq-aaa} and \eqref{eq-bbb} gives
\[
\sup_{\xi\in \RR^{n+1},\,|\xi|=1} \big\langle \xi, (A-\gamma B) \xi\big\rangle_{\RR^{n+1}}
\le 
a R^{-\frac 12} -\gamma b <0 \text{ for } R\to+\infty,
\]
and the substitution into \eqref{eq-mm2} concludes the proof.

\subsection{Perturbations producing no eigenvalues}

\begin{lemma}\label{lem8}
Assume that the inequality \eqref{eq-lnun} is satisfied, then
for sufficiently small $\gamma>0$ the operator
$-\Delta^\Lambda_D-\gamma 1_C$ has exactly $N(\Lambda)$
eigenvalues in $(0,\nu)$.
\end{lemma}

\begin{proof} The perturbation potential is non-positive and with a compact support,
hence, it does not change the essential spectrum and
one has at least $N=N(\Lambda)$ eigenvalues in $(0,\nu)$.
Assume that there exists an $(N+1)$th eigenvalue, then by the min-max principle
it should satisfy
$\lambda_{N+1}(-\Delta^\Lambda_D-\gamma 1_C)\ge \lambda_{N+1}(A)$,
where $A$ in the operator $-\Delta-\gamma 1_C$ in $L^2(\Lambda)$
with the Dirichlet boundary condition at $\partial\Lambda$ and
an additional Neumann boundary condition at the both sides of
$\partial C\mathop{\cap}\partial\Lambda$. The operator
$A$ is unitarily equivalent
to $(-\Delta^C_{DN}-\gamma)\oplus A_1\dots \oplus A_n$,
where each $A_j$ is the Laplacian in $L^2(\RR_+\times\omega_j)$
with the Dirichlet boundary condition at $(\partial \omega_j)\times \RR_+$
and with the Neumann boundary condition at $\omega_j\times\{0\}$,
and by the separation of variables one has $\spec (A_j)=[\nu_j,+\infty)$
and $A_j\ge \nu$.
Therefore, $\lambda_{N+1}(A)=\lambda_{N+1}(-\Delta^C_{DN}-\gamma)=
\lambda_{N+1}(-\Delta^C_{DN})-\gamma$, and
$\lambda_{N+1}(-\Delta^\Lambda_D-\gamma 1_C)
\ge \lambda_{N+1}(-\Delta^C_{DN})-\gamma$.
By \eqref{eq-lnun}, for sufficiently small $\gamma$
the right-hand side is still greater than $\nu$,
while the left-hand side is strictly less than $\nu$,
which is a contradiction.
\end{proof}

\section{Examples}\label{sexam}

Due to a large number of possible examples, cf.~\cite{naz12},
we restrict our attention to the configurations
for which either a particularly explicit result
or an improvement of previous studies can be presented.

\subsection{Rounded corner}

As one of the simplest examples one can consider the configuration $\Lambda$
consisting of two copies of the half-strip $\RR_+\times (0,1)$
attached to the flat sides of a circular sector $C$ of unit radius
and of opening $\alpha\in(0,\pi)$, see Figure~\ref{fig-r}(a).
In the polar coordinates
$(r,\theta)$ one has $C:=\big\{ (r,\theta): r\in(0,1),\, \theta\in(0,\alpha) \big\}$.
The cross-section is $\omega=(0,1)$ with $\nu=\pi^2$.

\begin{figure}
\centering
\begin{minipage}{65mm}
\begin{center}
\includegraphics[height=40mm]{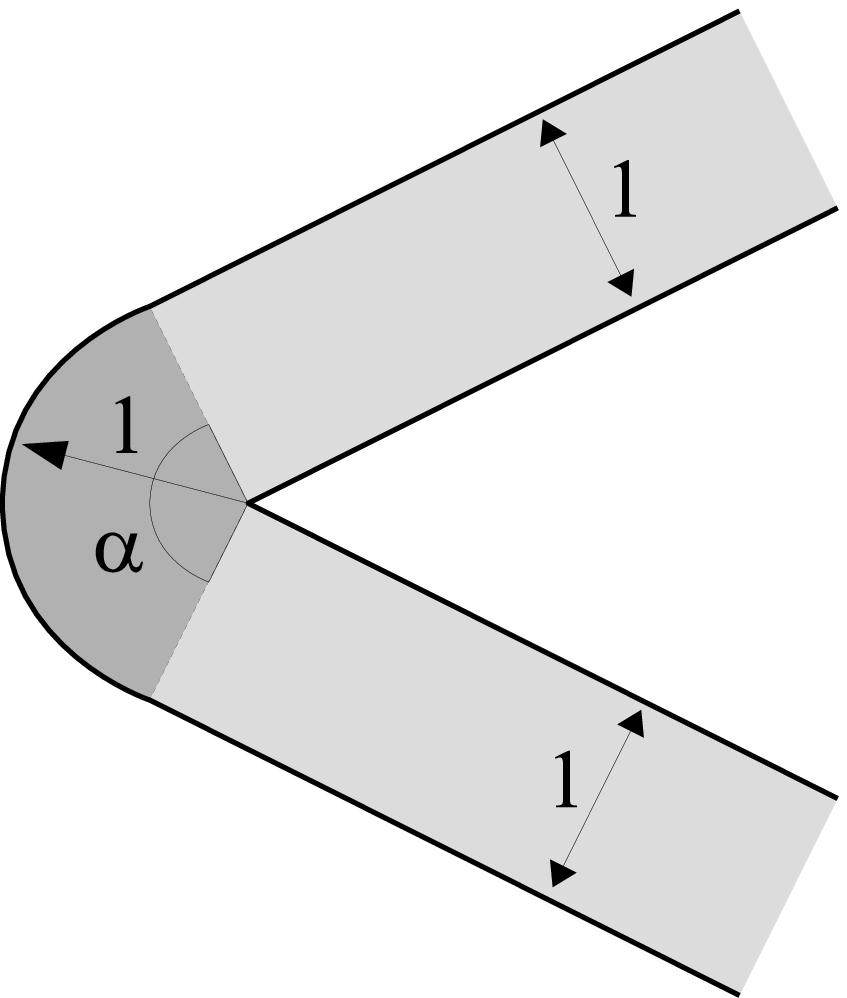}\\
(a)
\end{center}
\end{minipage}
\begin{minipage}{65mm}
\begin{center}
\includegraphics[height=40mm]{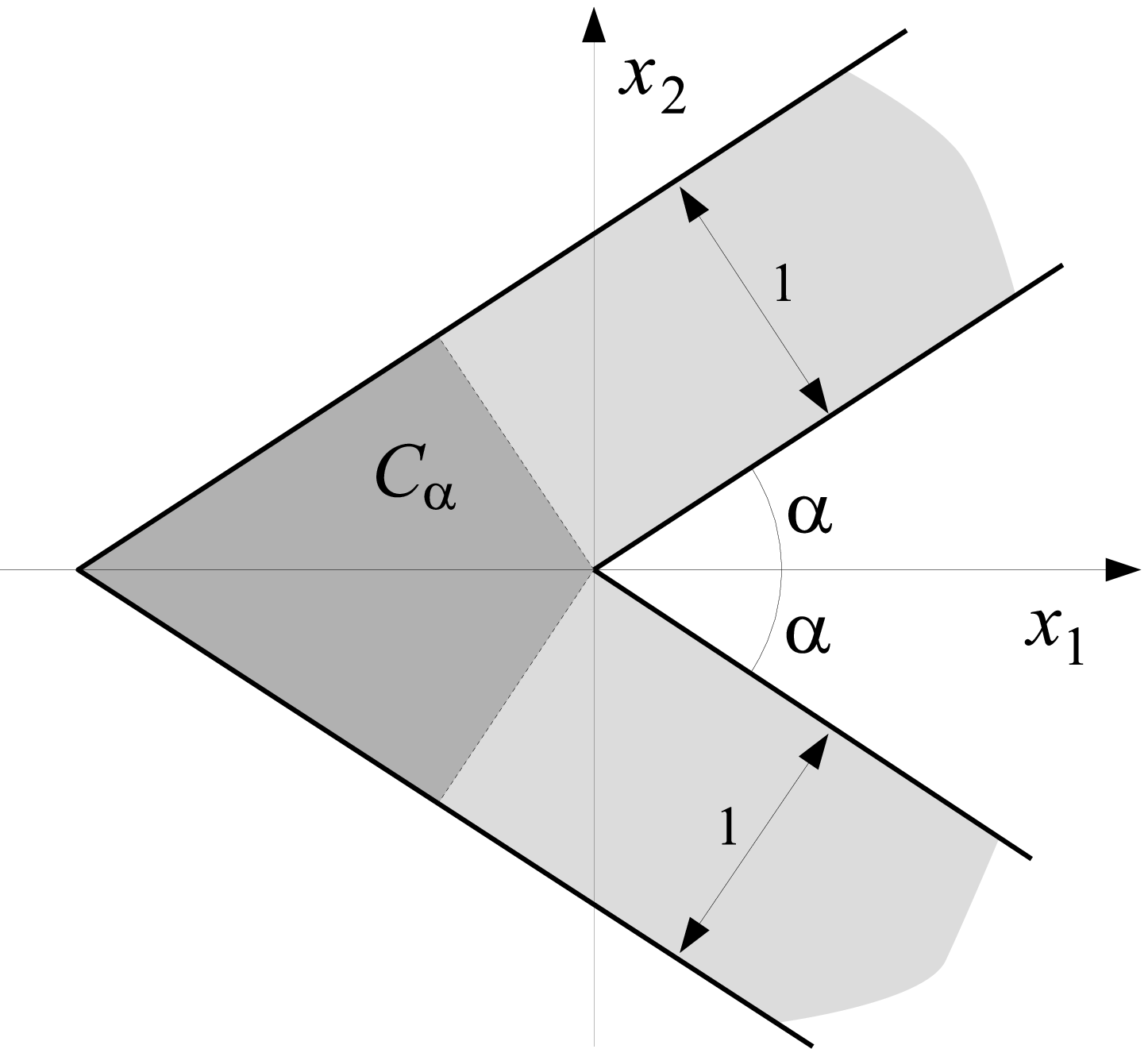}\\
(b)
\end{center}
\end{minipage}
\caption{(a) Waveguide with a rounded corner. The center $C$ is dark-shaded.
(b) Broken waveguide $\Lambda_\alpha$ with a dark-shaded center $C_\alpha$.\label{fig-r}}
\end{figure}

\begin{prop}
For any $\alpha\in(0,\pi)$, the operator $-\Delta^\Lambda_D$
has a single discrete eigenvalue and no threshold resonance.
\end{prop}

\begin{proof} The existence of at least one eigenvalue
follows from the general results for curved waveguides of constant
width~\cite{gold}.
The associated operator $-\Delta^C_{DN}$ admits a separation of variables
in polar coordinates, and the eigenvalues are the numbers
$\lambda_{n,k}:=\big(j_{\frac{\pi n}{\alpha},k}\big)^2$,
$n\in\NN\mathop{\cup}\{0\}$, $k\in \NN$, where
$j_{s,k}$ is the $k$th zero of the Bessel function $J_s$.
Recall, see e.g. \cite{34}, that we have
the inequalities $j_{s,k}>s + k\pi -\frac{1}{2}$ for $s> \frac{1}{2}$
and $j_{s,k}>s + k\pi -\frac{\pi}{2}+\frac{1}{2}$ 
for $s> -\frac{1}{2}$, and it follows that
$\lambda_{n,k}>\nu$ for $(n,k)\ne (0,1)$.
As the lowest eigenvalue $\lambda_{0,1}$
is simple, the result follows by Corollary~\ref{cor1}.
\end{proof}

\subsection{Broken waveguide}\label{ss-brok}

Consider the domain
\[
\Lambda_\alpha=\Big\{(x_1,x_2): \dfrac{\cos \alpha}{\sin\alpha}\, |x_2|-\dfrac{1}{\sin\alpha}
<x_1< \dfrac{\cos \alpha}{\sin\alpha} |x_2|\Big\},
\quad
\alpha\in\Big(0,\frac{\pi}{2}\Big).
\]
The domain can be considered
as two copies of the half-strip $\RR_+\times(0,1)$
attached to a quadrangle $C_\alpha$ having a symmetry axis,
see Figure~\ref{fig-r}(b). As in the previous example, $\nu=\pi^2$.
It is known since a long time, cf. \cite{avi}, that
the discrete spectrum is always non-empty, that
each discrete eigenvalue is monotonically increasing
with respect to $\alpha$, that the number of the eigenvalues
increases infinitely as $\alpha$ approaches $0$, and the eigenvalue
asymptotics for small $\alpha$ is computed in~\cite{dr2}.
A very detailed discussion can be found in \cite{dr1}.
We would like to improve the existing results as follows.

\begin{prop}\label{prop-br}
For $\alpha \in \big(\arctan \frac{\sqrt 3}{4},\frac{\pi}{2} \big)$
the operator $-\Delta^{\Lambda_\alpha}_D$
has a single discrete eigenvalue and no threshold resonance.
\end{prop}

\begin{figure}
\centering
\begin{tabular}{cc}
\begin{minipage}[c]{90mm}
\begin{center}
\includegraphics[height=30mm]{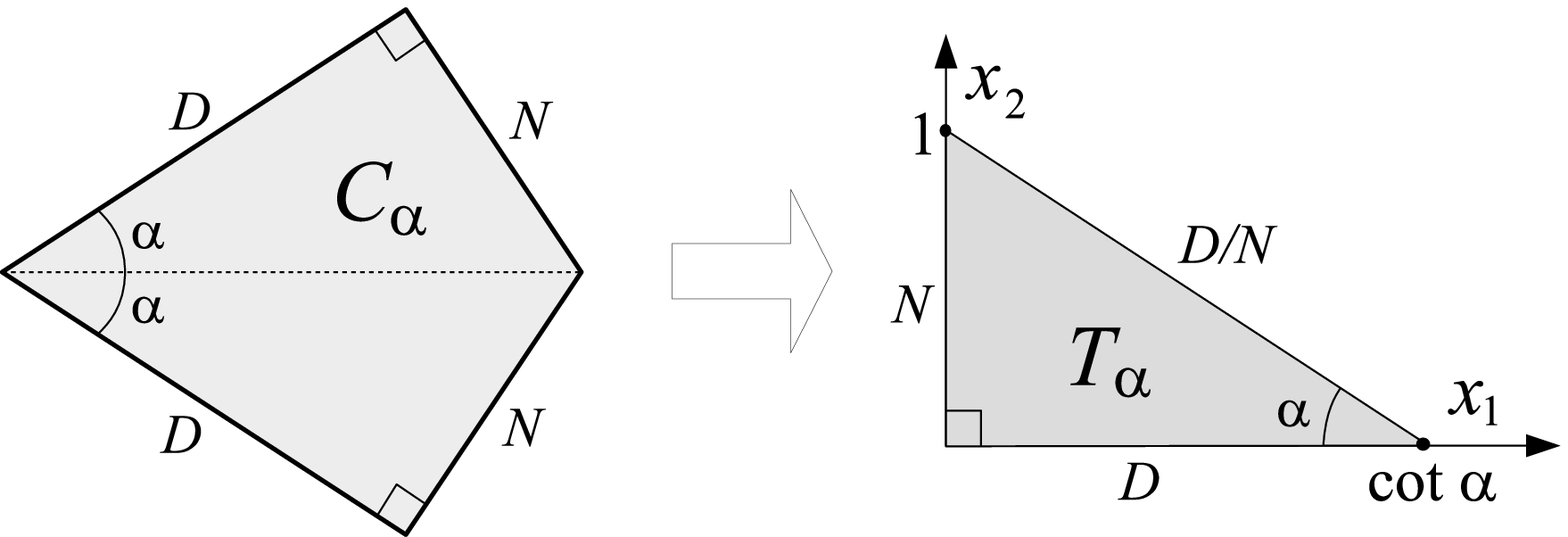}
\end{center}
\end{minipage}
&
\begin{minipage}[c]{50mm}
\begin{center}
\includegraphics[height=30mm]{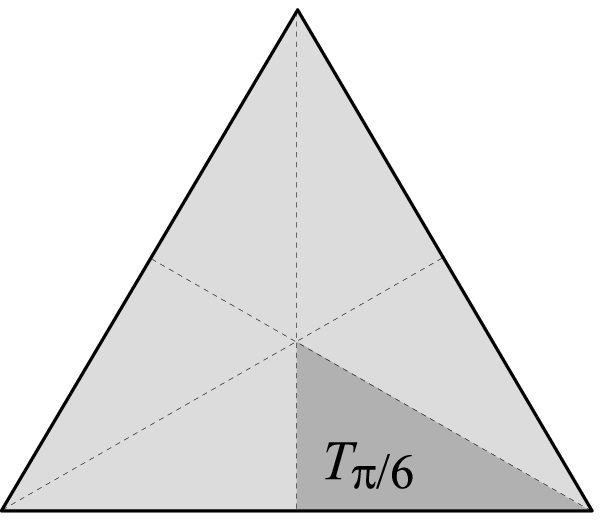}
\end{center}
\end{minipage}\\
(a) & (b)
\end{tabular}
\caption{(a) The quadrangle $C_\alpha$ and the triangle $T_\alpha$. The symbols $D/N$ correspond
to the Dirichlet/Neumann boundary condition.
 (b) The equilaterial triangle~$\Omega$.\label{fig-br}}
\end{figure}

\begin{proof}
In view of Corollary~\ref{cor1} it is sufficient to show that $\lambda_2(-\Delta^{C_\alpha}_{DN})>\pi^2$ for $\alpha$ in the interval indicated.
The decomposition of $C_\alpha$ with respect to the horizontal symmetry axis
shows that $-\Delta^{C_\alpha}_{DN}$ is unitarily equivalent
to $A^D_\alpha\oplus A^N_\alpha$, where $A^{D/N}_\alpha$ are the Laplacians
on the right-angled triangle
$T_\alpha:= \big\{
(x_1,x_1):\, 0<x_2<1- x_1 \tan \alpha
\big\}$ with the Dirichlet boundary condition on the bottom side,
with the Neumann boundary condition on the left side
and with the Dirichlet/Neumann boundary condition
on the hypotenuse, see Figure~\ref{fig-br}(a).
Denote
\[
U:=\big\{u\in C^\infty (\overline{T_\alpha}):
u(x_1,x_2)=0 \text{ for $x_1=0$ or $x_2=1- x_1 \tan \alpha$}\big\},
\]
then $\lambda_1(A^D_\alpha)=\inf_{u\in U\setminus\{0\}}\|\nabla u\|^2_{L^2(T_\alpha)}/ \|u\|^2_{L^2(T_\alpha)}$.
Furthermore, for $u\in U$ we have the
one-dimensional inequalities
\begin{align*}
\int_{T_\alpha} \Big(\dfrac{\partial u}{\partial x_1}\Big)^2\dd x
&=\int_0^1 \int_0^{(1- x_2) \cot \alpha} 
\Big(\dfrac{\partial u}{\partial x_1}\Big)^2\dd x_1\,\dd x_2\\
&\ge 
\int_0^1 \dfrac{\pi^2}{4 (1- x_2)^2 \cot^2 \alpha}\int_0^{(1- x_2) \cot \alpha} 
 u(x_1,x_2)^2\dd x_1\,\dd x_2,\\
\int_{T_\alpha} \Big(\dfrac{\partial u}{\partial x_2}\Big)^2\dd x
&=\int_0^{\cot \alpha} \int_0^{1- x_1 \tan \alpha} 
\Big(\dfrac{\partial u}{\partial x_2}\Big)^2\dd x_2\,\dd x_1\\
&\ge
\int_0^{\cot \alpha} \dfrac{\pi^2}{(1- x_1 \tan \alpha)^2}\int_0^{1- x_1 \tan \alpha} 
u(x_1,x_2)^2\dd x_2\,\dd x_1,
\end{align*}
hence,
\[
\|\nabla u\|^2_{L^2(T_\alpha)}
\ge \pi^2\int_{T_\alpha }
\bigg[
\dfrac{\tan^2\alpha}{4 (1- x_2)^2}
+
\dfrac{1}{(1- x_1 \tan \alpha)^2}
\bigg]\, u^2 \,\dd x
\ge \pi^2\Big(\dfrac{1}{4}\, \tan^2 \alpha +1\Big)
\|u\|^2_{L^2(T_\alpha)}.
\]
Therefore, $A^D_\alpha>\pi^2$ for any $\alpha$, and it remains to find a condition
guaranteeing that $\lambda_2(A^N_\alpha)>\pi^2$.

Let us study now the operator $A^N_{\frac \pi 6}$. Remark that any eigenfunction of $A^N_{\frac \pi 6}$
can be extended, using the symmetries with respect to the Neumann sides,
to a Dirichlet eigenfunction of the equilaterial triangle $\Omega$
with side length $2\sqrt 3$, see Figure~\ref{fig-br}(b). Therefore, for any $k\in \NN$
we have $\lambda_k(A^N_{\frac \pi 6})\ge \lambda_k(-\Delta^\Omega_{D,s})$,
where $-\Delta^\Omega_{D,s}$ is the restriction of the Dirichlet Laplacian $-\Delta^\Omega_D$ in $\Omega$
to the functions satisfying the Dirichlet boundary conditions,
symmetric with respect to the medians and invariant under the rotations by $\frac{2\pi}{3}$ around the center of the triangle.
Recall that the eigenvalues and the eigenfunctions of the Dirichlet Laplacian
on the equilateral triangles are known explicitly, see e.g.~\cite{prag}, and the eigenvalues
of $-\Delta^\Omega_D$ are the numbers
\[
\mu_{m,n}=\dfrac{4 \pi^2}{27} \big(m^2+mn+n^2), \quad (m,n)\in\NN\times\NN.
\]
The eigenfunction corresponding to the first eigenvalue
$\mu_{1,1}$ belongs to the domain of $-\Delta^\Omega_{D,s}$,
hence, $\lambda_1(-\Delta^\Omega_{D,s})=\frac{4\pi^2}{9}$.
On the other hand, one has $\lambda_2(-\Delta^\Omega_D)=\lambda_3(-\Delta^\Omega_D)=\mu_{1,2}\equiv\mu_{2,1}$, but no associated
eigenfunction has the required symmetries:
there is just one eigenfunction symmetric with respect to one of medians, but
it is not rotationally invariant.
Hence, $\lambda_2(A^N_{\pi/6})\ge \lambda_2(-\Delta^\Omega_{D,s})\ge \lambda_4(-\Delta^\Omega_{D})=\mu_{2,2}=\frac{16\pi^2}{9}>\pi^2$.

Note that the map $\Phi_{\alpha,\beta}:L^2(T_\alpha)\to L^2(T_\beta)$ given by
\[
\big(\Phi_{\alpha,\beta} u\big)(x_1,x_2)=u \Big( \dfrac{\cot \alpha}{\cot \beta}x_1,x_2\Big)
\]
is bijective from the form domain of $A^N_\alpha$ to that of $A^N_\beta$, and
\[
\dfrac{\| \nabla\Phi_{\alpha,\beta} u\|^2_{L^2(T_\beta)}}{\|\Phi_{\alpha,\beta} u\|^2_{L^2(T_\beta)}}
=\dfrac{\displaystyle\int_{T_\alpha} \bigg[\Big(\dfrac{\tan \beta}{\tan\alpha}\Big)^2 \Big(\dfrac{\partial u}{\partial x_1}\Big)^2
+\Big(\dfrac{\partial u}{\partial x_2}\Big)^2\bigg]\dd x
}{\|u\|_{L^2(T_\alpha)}},
\]
and it follows by the min-max principle that
\begin{equation}
        \label{eq-lab1}
\lambda_k(A^N_\beta)\ge \min \bigg\{\Big(\dfrac{\tan \beta}{\tan\alpha}\Big)^2,1\bigg\} \lambda_k(A^N_\alpha),  \quad
\quad k \in \NN.
\end{equation}
Hence, for $\alpha\ge \frac{\pi}{6}$ we obtain $\lambda_2(A^N_\alpha)\ge \lambda_2(A^N_{\frac \pi 6})>\pi^2$,
while for $\alpha<\frac{\pi}{6}$ we arrive at
\[
\lambda_2(A^N_\alpha) \ge \Big(\dfrac{\sqrt 3}{\cot \alpha}\Big)^2 
\lambda_2(A^N_{\frac \pi 6})= \dfrac{16 \pi^2}{3}\,\tan^2\alpha, 
\]
and $\lambda_2(A^N_\alpha)>\pi^2$ for $\tan \alpha >\frac{\sqrt 3}{4}$.
\end{proof}

Remark that our lower bound $\arctan \frac{\sqrt{3}}{4}\simeq 0.409\simeq 23.4^\circ$
for the existence of a unique discrete eigenvalue
improves the previously known value $\arctan \sqrt{0.4}\simeq 0.564\simeq 32.3^\circ$
obtained in \cite{naz12s}.
Anyway, our estimate is not expected to be optimal: the numerical
simulations~\cite{bind,naz12s} suggest that the second eigenvalue appears
for $\alpha\simeq 0.242\simeq 13.7^\circ$.

Note that in this specific example a more detailed result can obtained using
the monotonicity of the eigenvalues with respect to the angle.
Namely, denote $\cN(\alpha):=N(\Lambda_\alpha)$ the number of the discrete eigenvalues,
the function $\cN$ is then piecewise constant and non-increasing,
and $\cN(\alpha)$ tends to $\infty$ as $\alpha$ approaches $0$.
Hence, there exists an infinite sequence $\frac{\pi}{2}=\alpha_0>\alpha_1>\alpha_2>\dots$
such that $\cN$ is constant on each interval
$[\alpha_{n},\alpha_{n-1})$ but has a jump at each $\alpha_n$,  $n\in \NN$,
and $\alpha_1\le\arctan \frac{\sqrt{3}}{4}$ by Proposition~\ref{prop-br}.
A  modification of the proof of Theorem~\ref{thm1} presented in 
Appendix~\ref{proof10} gives then the following result:
\begin{prop}\label{prop10}
Assume that $\Lambda_\alpha$ admits a threshold resonance
for some $\alpha\in \big(0, \frac\pi 2\big)$,
then the counting function $\cN$ has a jump at $\alpha$.
\end{prop}

In other words, there is just a discrete (but infinite) family of critical angles for which the existence of threshold resonances
is possible. Remark that such a picture is typical
for problems with threshold resonances, cf. \cite{rauch},
and it appears in other problems governed by geometric parameters,
see e.g.~\cite{bor1,bor2,naz-t}.

\subsection{T- and Y-junctions}

\begin{figure}
\centering
\begin{minipage}{65mm}
\begin{center}
\includegraphics[height=20mm]{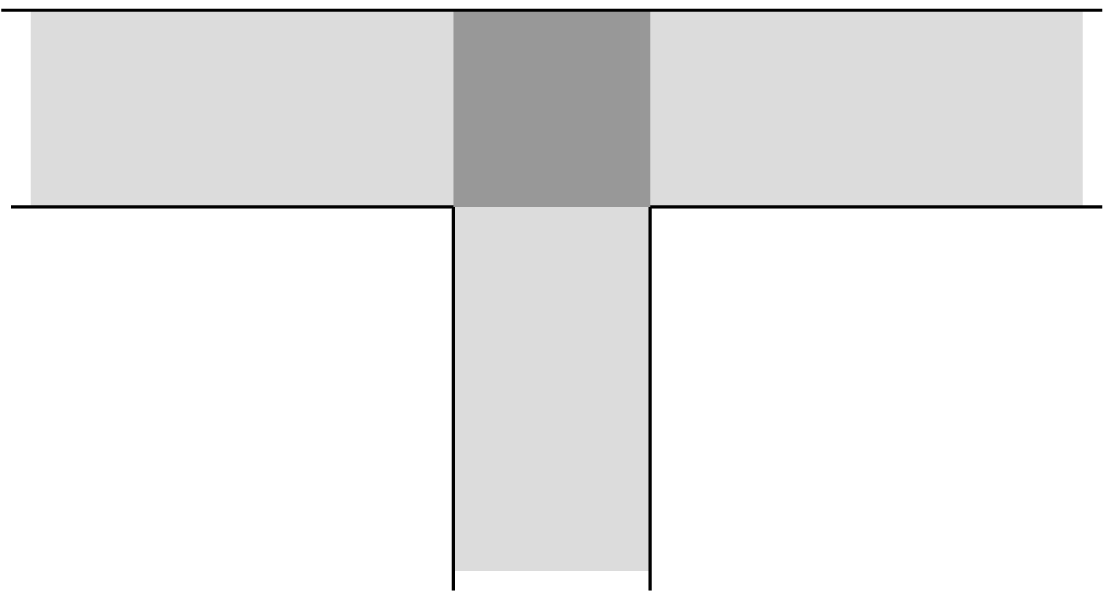}\\
(a)
\end{center}
\end{minipage}
\begin{minipage}{65mm}
\begin{center}
\includegraphics[height=25mm]{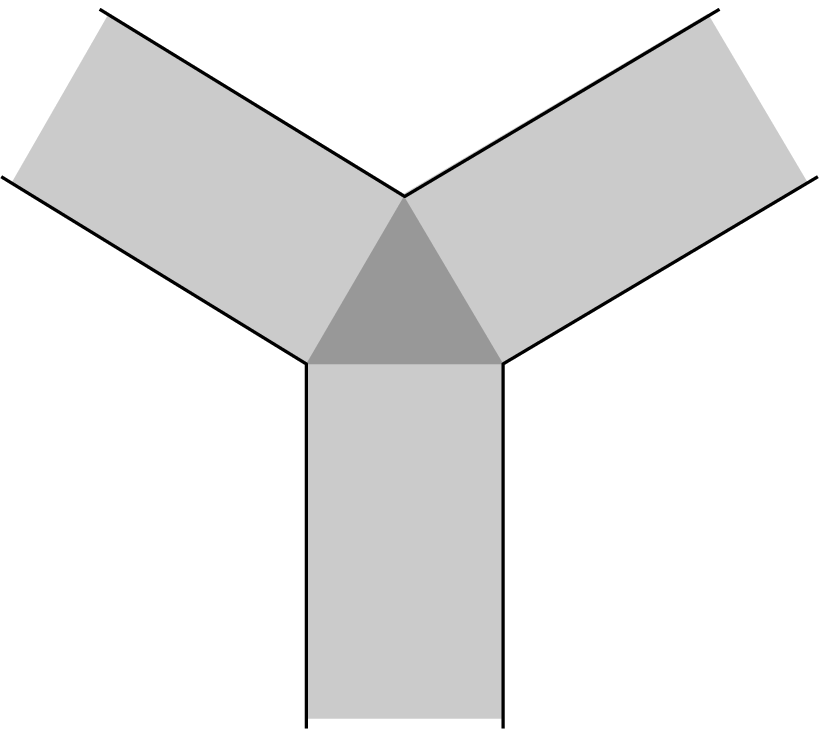}\\
(b)
\end{center}
\end{minipage}
\caption{(a) $T$-shaped waveguide $\Lambda_T$.
(b) $Y$-shaped waveguide $\Lambda_Y$.\label{fig-ty}}
\end{figure}

The $T$-junction $\Lambda_T$ represents three copies of the half-strip
$\RR_+\times(0,1)$ attached to three sides of a unit square, while
the $Y$-junction $\Lambda_Y$ is obtained from three copies of the same half-strip 
attached
to the three sides of an equalateral triangle of unit side length,
see Figure~\ref{fig-ty}, and the absense of threshold resonances
for the two configurations was already obtained in \cite{naz-t,naz-hex}.
For illustrative purposes, let us repeat the respective constructions.
For the both cases we have $\nu=\pi^2$, and the presence
of the discrete spectrum follows from the domain monotonicity
by comparing with the broken waveguides (see subsection~\ref{ss-brok})
with $\alpha=\frac{\pi}{4}$ for $\Lambda_T$ and $\alpha=\frac{\pi}{3}$ for $\Lambda_Y$.
For $\Lambda_T$, the operator $-\Delta^C_{DN}$ is the Laplacian on the unit square
with the Dirichlet boundary condition on one side and the Neumann boundary condition
on the other three sides. The separation of variables shows that
$\lambda_2(-\Delta^C_{DN})=\frac{5\pi^2}{4}>\nu$, and  Corollary~\ref{cor1} gives the result.
For $\Lambda_Y$, the operator $-\Delta^C_{DN}$ is the Neumann Laplacian
in the equilateral triangle of unit side length, and its second eigenvalue
is $\frac{16\pi^2}{9}>\nu$, see \cite{prag}, and we are again in the situation
of Corollary~\ref{cor1}.

Using a construction similar to the one used in the proof of Proposition~\ref{prop-br}
one can consider a more general class of domains starting either with $\Lambda_T$
or with $\Lambda_Y$. Namely, for $\theta\in \RR$ denote by $L_\theta$ the ray
$\RR_+ (\cos \theta,\sin \theta)$.
For $\alpha\in(0,\frac{\pi}{2}]$ consider the union of three rays
$Y_\alpha:=L_{-\frac \pi 2}\mathop{\cup}L_{\frac \pi 2 - \alpha}\mathop{\cup} L_{\frac \pi 2 + \alpha}$
and denote by $\Lambda_{Y,\alpha}$ its $\frac{1}{2}$-neighborhood, see Figure~\ref{fig-ya}.
Remark that for $\alpha=\frac{\pi}{3}$ and $\alpha=\frac{\pi}{2}$
we obtain respectively the above sets $\Lambda_Y$ and $\Lambda_T$.

\begin{prop}
Denote $\alpha_1:=\arccos (\sqrt{13}-3)\simeq 52,7^\circ$
and $\alpha_2:=\arctan\frac{4}{\sqrt{3}}\simeq 66,6^\circ$,
then for $\alpha\in (\alpha_1,\alpha_2)$ the Dirichlet
Laplacian in $\Lambda_{Y,\alpha}$ has a unique discrete eigenvalue and no threshold resonance.
\end{prop}

\begin{proof}
We are going to apply  Corollary~\ref{cor1} again. The existence of a non-empty
discrete spectrum follows again by comparing with the broken waveguides.
To study the eigenvalues $\lambda_2(-\Delta^C_{DN})$ we distinguish
between the cases $\alpha<\frac \pi3$
and $\alpha>\frac \pi 3$.

\begin{figure}[t]
\centering
\begin{tabular}{cc}
\begin{minipage}[c]{70mm}
\begin{center}
\includegraphics[height=45mm]{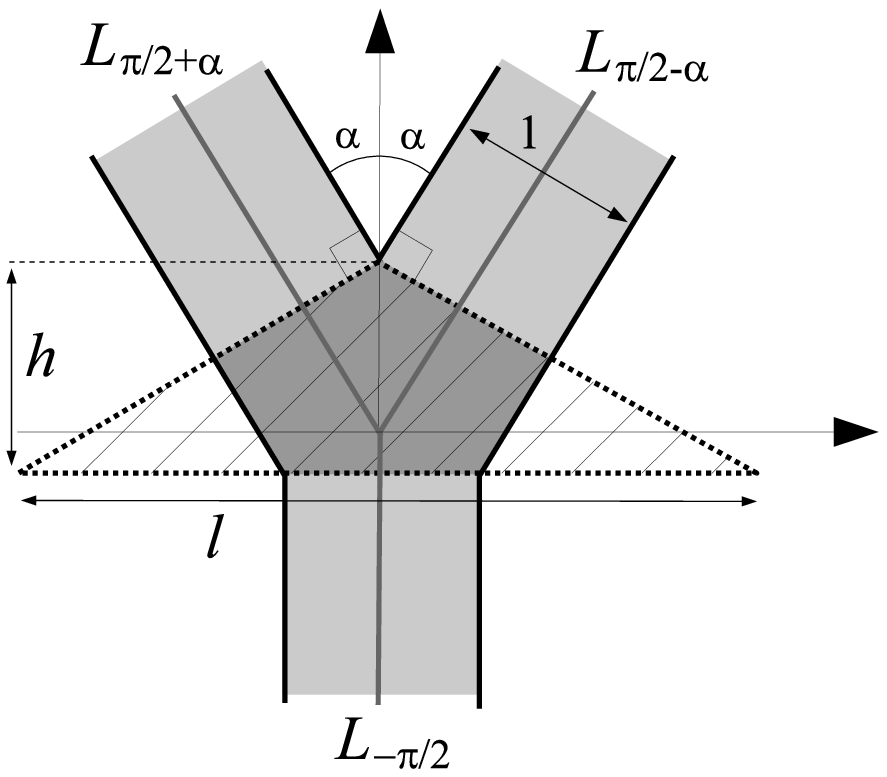}
\end{center}
\end{minipage}
&
\begin{minipage}[c]{70mm}
\begin{center}
\includegraphics[height=40mm]{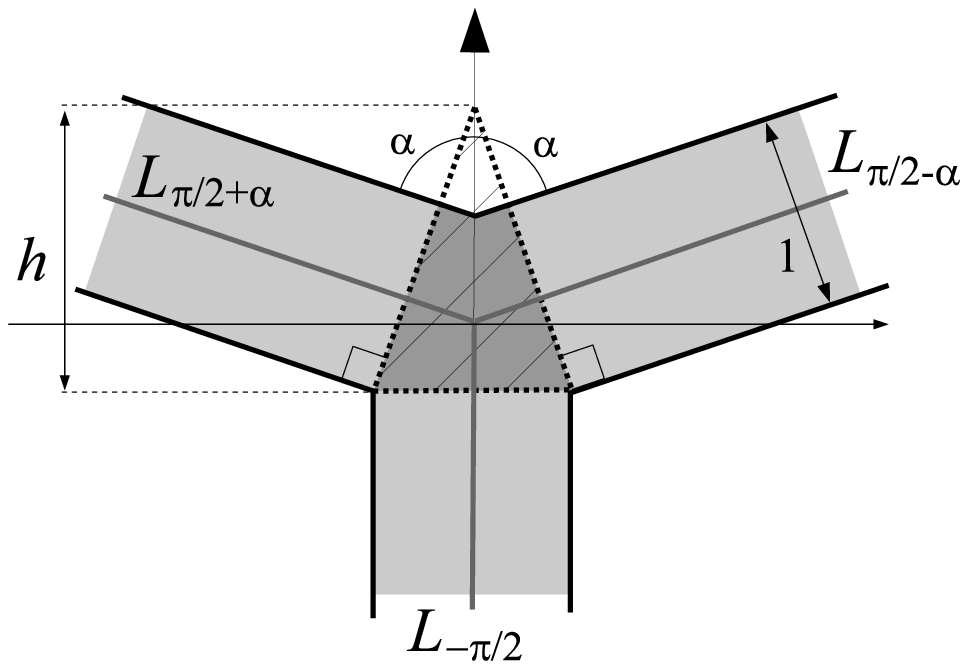}
\end{center}
\end{minipage}\\
(a) & (b)
\end{tabular}
\caption{The domain $\Lambda_{Y,\alpha}$ for (a) $\alpha<\frac \pi 3$ and (b) $\alpha>\frac \pi 3$. The center $C$
is dark-shaded and the triangle $M$ is hatched. \label{fig-ya}}
\end{figure}

Let $\alpha<\frac\pi 3$, then the smallest possible center $C$ is a convex pentagon.
By extending the three sides at which the Neumann boundary condition for $-\Delta^C_{DN}$
is imposed we obtain
an isosceles triangle $M$  with the base length $l$ and the height $h$  given by
\[
l=\dfrac{(2-\cos\alpha)\cos\alpha}{\sin^2\alpha},
\quad
h=\dfrac{2-\cos\alpha}{2\sin\alpha},
\]
see Figure~\ref{fig-ya}(a), and by the min-max principle we have the inequality
$\lambda_k(-\Delta^C_{DN})\ge \lambda_k(-\Delta^M_N)$, $k\in\NN$, where $-\Delta^M_N$
is the Neumann Laplacian in $M$.
Remark that $\frac{l}{h}=2\cot \alpha>\frac{2}{\sqrt 3}$, while the last value is
 the base/height ratio for the equilateral triangles. Therefore, by applying the  contraction
with the coefficient $\sqrt 3 \cot\alpha$
along the $x_1$-axis
 we obtain an equilaterial triangle $\Omega$ of height $h$, and,
similarly to \eqref{eq-lab1}, one has
\[
\lambda_k(-\Delta^M_N)\ge \Big(\dfrac{1}{\sqrt 3 \cot \alpha}\Big)^2\lambda_k(-\Delta^\Omega_N), \quad k\in\NN.
\]
As $\lambda_2(-\Delta^\Omega_N)=\frac{4\pi^2}{3h^2}$, see \cite{prag}, we arrive at
\[
\lambda_2(-\Delta^C_{DN})\ge\dfrac{16 \pi^2\sin^4 \alpha}{9\cos^2\alpha(2-\cos\alpha)^2}=:\lambda(\alpha),
\]
and solving the inequality $\lambda(\alpha)>\pi^2$ gives the
sought lower bound for $\alpha$.

Now let $\alpha>\frac \pi 3$, then the smallest possible center $C$
is a concave pentagon, and extending
the Neumann sides one obtains an isosceles triangle $M$ with a unit base and the height $h=\frac 12 \tan\alpha>\frac{\sqrt{3}}{2}$,
and the contraction along the $x_2$ axis with the coefficient $\frac{1}{\sqrt 3}\,\tan \alpha$ transforms $M$
into an equilateral triangle $\Omega_0$ of unit side length. As in  \eqref{eq-lab1} we have then
\[
\lambda_2(-\Delta^C_{DN})\ge \lambda_2(-\Delta^M_N)\ge \Big(\dfrac{\sqrt 3}{\tan\alpha}\Big)^2\lambda_2(-\Delta^{\Omega_0}_N)
=\dfrac{16 \pi^2 }{3 \tan^2 \alpha},
\]
and $\lambda_2(-\Delta^C_{DN})>\pi^2$ for $\tan \alpha< \frac{4}{\sqrt{3}}$, which gives the upper bound.
\end{proof}

\subsection{Crossing strips}

\begin{figure}
\centering
\centering
\begin{tabular}{cc}
\begin{minipage}[c]{65mm}
\begin{center}
\includegraphics[height=35mm]{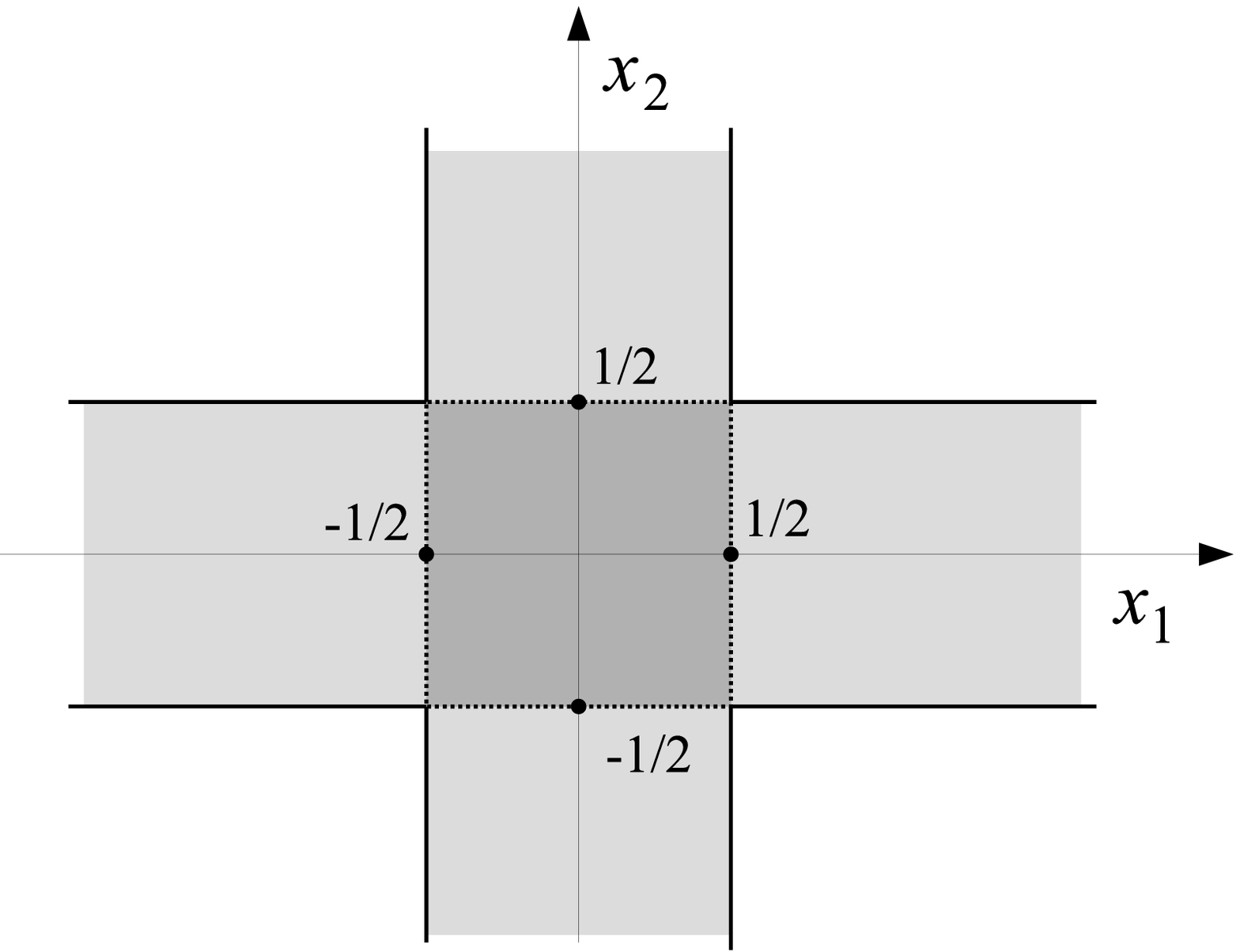}
\end{center}
\end{minipage}
&
\begin{minipage}[c]{65mm}
\begin{center}
\includegraphics[height=40mm]{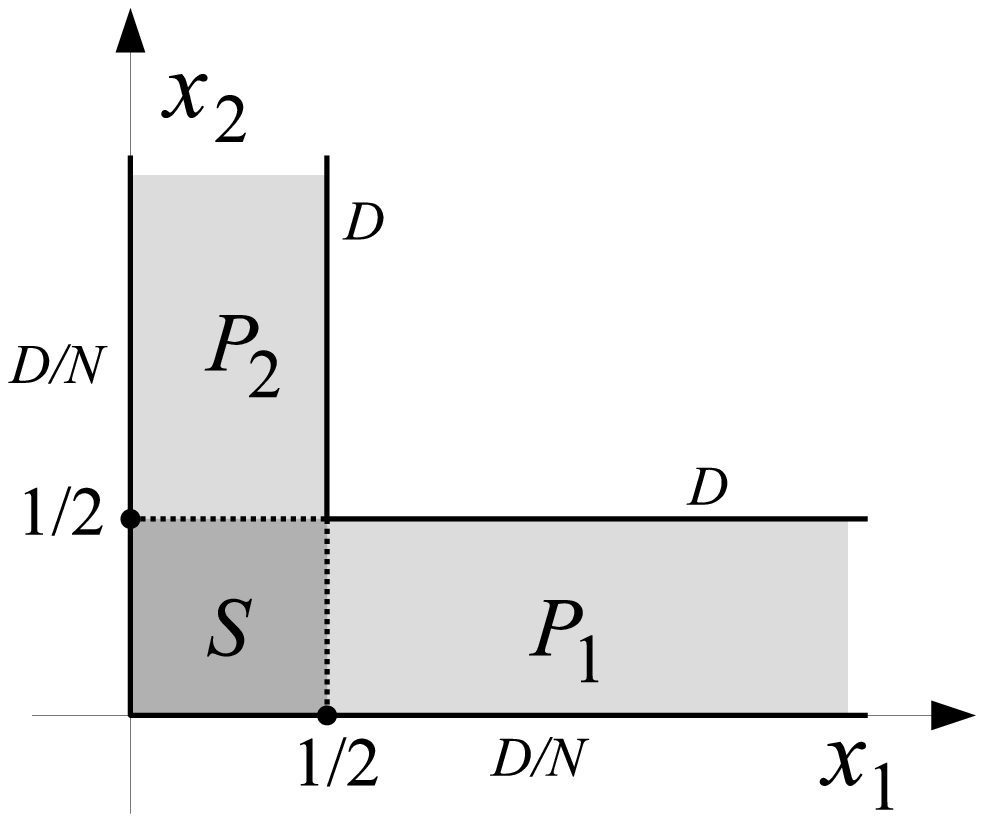}
\end{center}
\end{minipage}\\
(a) & (b)
\end{tabular}
\caption{ (a) The domain $\Lambda_\times$ with a dark-shaded center.
(b) The domain $\Pi$ decomposed into the square $S$ and two half-infinite strips $P_1$
and $P_2$. \label{fig-cross}}
\end{figure}

Consider the domain $\Lambda_\times:=\big((-\frac 1 2,\frac 1 2)\times\RR \big)\mathop{\cup}\big(
\RR\times(-\frac 1 2, \frac 1 2)\big)$, see Figure~\ref{fig-cross}(a). It can be viewed as four copies 
on the half-infinite strip $(0,1)\times\RR_+$ attached to the four sides of a unit square,
and we have again $\nu=\pi^2$.

\begin{prop}\label{prop13}
The Dirichlet Laplacian in $\Lambda_\times$ has a single discrete eigenvalue
and no threshold resonance.
\end{prop}

The rest of the subsection is dedicated to the proof.
As in the preceding examples, the existence of discrete eigenvalues
follows by comparing with broken waveguides.
Remark that the operator $-\Delta^C_{DN}$ is simply the Neumann Laplacian on the unit square,
and its second eigenvalue is $\pi^2=\nu$, and $-\Delta^\Lambda_D$
cannot have more than one discrete eigenvalue due to \eqref{eq-lnu2}. On the other hand, as the strict inequality
$\lambda_2(-\Delta^C_{DN})>\nu$
is not satisfied, the absence of threshold resonances does not follow directly
from Corollary~\ref{cor1}. We are going to show that the arguments can be
modified in order to cover $\Lambda_\times$.

Assume by contradiction that there is a non-trivial bounded solution $w$
to $-\Delta w=\pi^2 w$ in $\Lambda_\times$ vanishing at the boundary.
For $j,k\in\{0,1\}$ consider the functions $w_{jk}$ defined by
\[
w_{jk}(x_1,x_2)=w(x_1,x_2)+(-1)^j w(-x_1,x_2)+(-1)^k w(x_1,-x_2)
+(-1)^{j+k}w(-x_1,-x_2).
\]
Each of these four functions is a bounded solution to $-\Delta u=\pi^2 u$
in the domain $\Pi:=\big((0,\frac 12)\times\RR_+ \big)\mathop{\cup}\big(
\RR_+\times(0,\frac 12)\big)$, see Figure~\ref{fig-cross}(b),
vanishing at $\partial\Lambda_\times\mathop{\cap}\partial\Pi$
and satisfying the following boundary conditions at the remaining part of the boundary:
\begin{equation}
      \label{eq-ajk}
\begin{aligned}
(j,k)=(0,0):& \text{ Neumann on } \{0\}\times \RR_+ \text{ and } \RR_+\times\{0\},\\
(j,k)=(1,0):& \text{ Dirichlet on } \{0\}\times \RR_+ \text{ and Neumann on } \RR_+\times\{0\},\\
(j,k)=(0,1):& \text{ Neumann on } \{0\}\times \RR_+ \text{ and Dirichlet on } \RR_+\times\{0\},\\
(j,k)=(1,1):& \text{ Dirichlet on } \{0\}\times \RR_+ \text{ and } \RR_+\times\{0\}.
\end{aligned}
\end{equation}
Furthermore, at least one of $w_{jk}$ is not identically zero.
Let $A_{jk}$ be the Laplacian in $L^2(\Pi)$ with the Dirichlet boundary condition
at $\partial\Lambda_\times\mathop{\cap}\partial\Pi$ and with the boundary
conditions~\eqref{eq-ajk} on $\partial\Pi\setminus\partial\Lambda_\times$ and denote
by $N_{jk}$ the number of discrete eigenvalues of $A_{jk}$ in $(0,\pi^2)$.
The Dirichlet Laplacian in $\Lambda_\times$
is then unitarily equivalent to the direct sum of $A_{jk}$,
and one has $\sum_{j,k=0}^1 N_{jk}=N(\Lambda_\times)=1$.
Proceeding literally as in Lemma~\ref{lem5} one proves
the following assertion:

\begin{lemma}\label{lem-apr}
If $w_{jk}$ is not identically zero, then for any
non-empty bounded open subset $\Omega$ of $\Pi$ and any
$\gamma>0$ the operator  $A_{jk}-\gamma 1_\Omega$
has at least $N_{jk}+1$ eigenvalues in $(0,\pi^2)$.
\end{lemma}

In addition, denote by $A^N_{jk}$ the Laplacian in $L^2(\Pi)$ with the same boundary
condition as $A_{jk}$ and an additional Neumann boundary condition
at the lines $x_1=\frac 1 2$ and $x_2=\frac 1 2$, i.e. on the dash lines in Figure~\ref{fig-cross}(b),
then $A_{jk}\ge A^N_{jk}$.
Furthermore, $A^N_{jk}=M^0_{jk}\oplus M^1_{jk}\oplus M^2_{jk}$,
where $M^0_{jk}$, $M^1_{jk}$, $M^2_{jk}$ are Laplacians with suitable
boundary conditions in respectively
the square $S:=(0,\frac 1 2)^2$ and the half-strips $P_1:=(\frac 12,\infty)\times(0,\frac 1 2)$
and $P_2:=(0,\frac 1 2)\times(\frac 12,\infty)$, and each
$M^s_{jk}$ admits a separation of variables.
Due to the inequality $A_{jk}-\gamma 1_\Omega\ge A^N_{jk}-\gamma 1_\Omega$
it is sufficient to construct, for each combination $(j,k)$,
an non-empty bounded open set $\Omega_{jk}\subset \Pi$
such that
\begin{equation}
   \label{eq-crr}
\text{	$A^N_{jk}-\gamma 1_{\Omega_{jk}}$ has exactly $N_{jk}$
eigenvalues in $(0,\pi^2)$ as $\gamma>0$ is sufficiently small.}
\end{equation}
Let $(j,k)=(1,1)$, then $M^0_{11}\ge 2\pi^2$, $M^1_{11}\simeq M^2_{11}\ge 4\pi^2$,
and $A^N_{11}\ge 2\pi^2$, hence, $N_{11}=0$. Therefore, any $\Omega_{11}\subset \Pi$
satisfies \eqref{eq-crr}. For $(j,k)=(1,0)$ we have $M^0_{10}\ge \pi^2$,
$M^1_{10}\ge \pi^2$, $M^2_{10}\ge 4\pi^2$, $N_{10}=0$ and
Eq.~\eqref{eq-crr} is satisfied for any $\Omega_{10}\subset P_2$.
In the same way, $N_{01}=0$, and Eq.~\eqref{eq-crr} holds for $(j,k)=(0,1)$
with any $\Omega_{01}\subset P_1$.
Finally, for $(j,k)=(0,0)$ we have $N_{11}=0$,
$M^1_{00}\simeq M^2_{00}\ge \pi^2$
and $\lambda_2(M^0_{00})=4\pi^2 >\pi^2$.
Therefore, Eq.~\eqref{eq-crr} holds with $\Omega_{00}=S$.

The combination of Lemma~\ref{lem-apr} with \eqref{eq-crr}
gives Proposition~\ref{prop13}.

\subsection{Configuration with several discrete eigenvalues}\label{ssmany}

The main difficulty in the use of Theorem~\ref{thm1} is that
it requires the exact knowledge of the quantity $N(\Lambda)$.
The analysis of the preceding examples was covered
by Corollary~\ref{cor1} due to the equality $N(\Lambda)=1$.
Let us give an example of a configuration $\Lambda$ with $N(\Lambda)=2$
for which the application of Theorem~\ref{thm1} is still possible.

For $a> 0$ and $b>2$, denote $\Pi_{a,b}:=(0,a)\times (0,b)$.
Let $\Lambda\equiv\Lambda_{a,b}$ be the star waveguide obtained by attaching
two copies of the half-strip $\RR_+\times (0,1)$
to a side of length $b$ of $\Pi_{a,b}$, see Figure~\ref{figell}(a).
The exact position of the two branches
along the side is not important, they
are only assumed non-intersecting. We have obviously $\nu=\pi^2$.

Take as a center $C:=\Pi_{a,b}$.
Let $A$ be the Laplacian in $\Pi_{a,b}$ with the Neumann
boundary condition on a side of length $b$
and with the Dirichlet boundary conditions on the other three sides.
Furthermore, let $B$ be the Dirichlet Laplacian in $\Pi_{a,b}$. 
Using the min-max principle we have then the following observations:
\begin{itemize}
\item if for some $j\in\NN$ one has $\lambda_j(B)<\nu$, then $N(\Lambda)\ge j$,
\item for any $j\in\big\{1,\dots,N(\Lambda)\big\}$ one has
$\lambda_j(A)\le \lambda_j(-\Delta^\Lambda_D)\le \lambda_j(B)$,
\item for any $j\in\NN$ one has
$\lambda_j(A)\le \lambda_j(-\Delta^C_{DN})$,
\end{itemize}
and a simple application of Theorem~\ref{thm1} gives the following assertion:
\begin{lemma}\label{lem15}
If for some $n\in\NN$ one has the strict inequalities
$\lambda_n(A)<\nu<\lambda_{n+1}(A)$  and  $\lambda_n(B)<\nu$,
then $N(\Lambda)=n$ and $\Lambda$ has no threshold resonance.
\end{lemma}
The operators $A$ and $B$ admit a separation of variables, and their eigenvalues
are the numbers
\begin{align*}
\mu_{m,n}(A):&=\pi^2 \Big( \dfrac{(2m-1)^2}{4a^2}+ \dfrac{n^2}{b^2}\Big), &m,n\in\NN,\\
\mu_{m,n}(B):&=\pi^2 \Big( \dfrac{m^2}{a^2}+ \dfrac{n^2}{b^2}\Big), & m,n\in\NN,
\end{align*}
respectively, enumerated in the non-decreasing order. Therefore, the following result holds:
\begin{prop}\label{propell}
Let $a$ and $b$ satisfy the inequalities
\begin{equation}
 \label{eq-aba}
a>0, \quad b>2,
\quad
\dfrac{4}{a^2}+\dfrac{1}{b^2}<1<\dfrac{25}{4a^2}+\dfrac{1}{b^2},
\quad
1< \dfrac{1}{4a^2}+\dfrac{4}{b^2},
\end{equation}
then the Dirichlet Laplacian in $\Lambda$ has exactly two discrete eigenvalues
and no threshold resonance.
\end{prop}

\begin{proof}
The inequalities \eqref{eq-aba} can be rewritten as $\mu_{2,1}(B)<\nu<\mu_{3,1}(A)$
and $\mu_{1,2}(A)>\nu$. As for $(j,k)\in\NN\times\NN$ there holds $\mu_{j,k}(A)<\mu_{j,k}(B)$,
we arrive at $\lambda_2(A)=\mu_{2,1}(A)<\nu$ and $\lambda_2(B)=\mu_{2,1}(B)<\nu$
together with $\lambda_3(A)>\nu$, and the result follows from Lemma~\ref{lem15} with $n=2$.
\end{proof}

\begin{figure}
\centering
\centering
\begin{tabular}{cc}
\begin{minipage}[c]{50mm}
\begin{center}
\includegraphics[height=18mm]{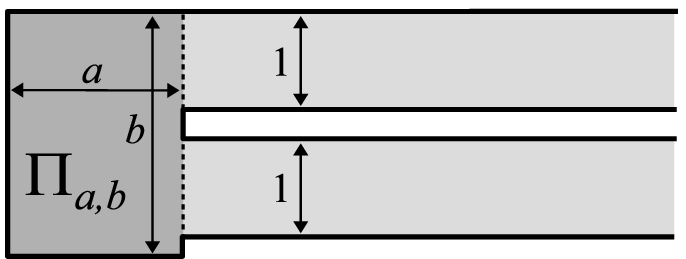}
\end{center}
\end{minipage}
&
\begin{minipage}[c]{80mm}
\begin{center}
\includegraphics[height=50mm]{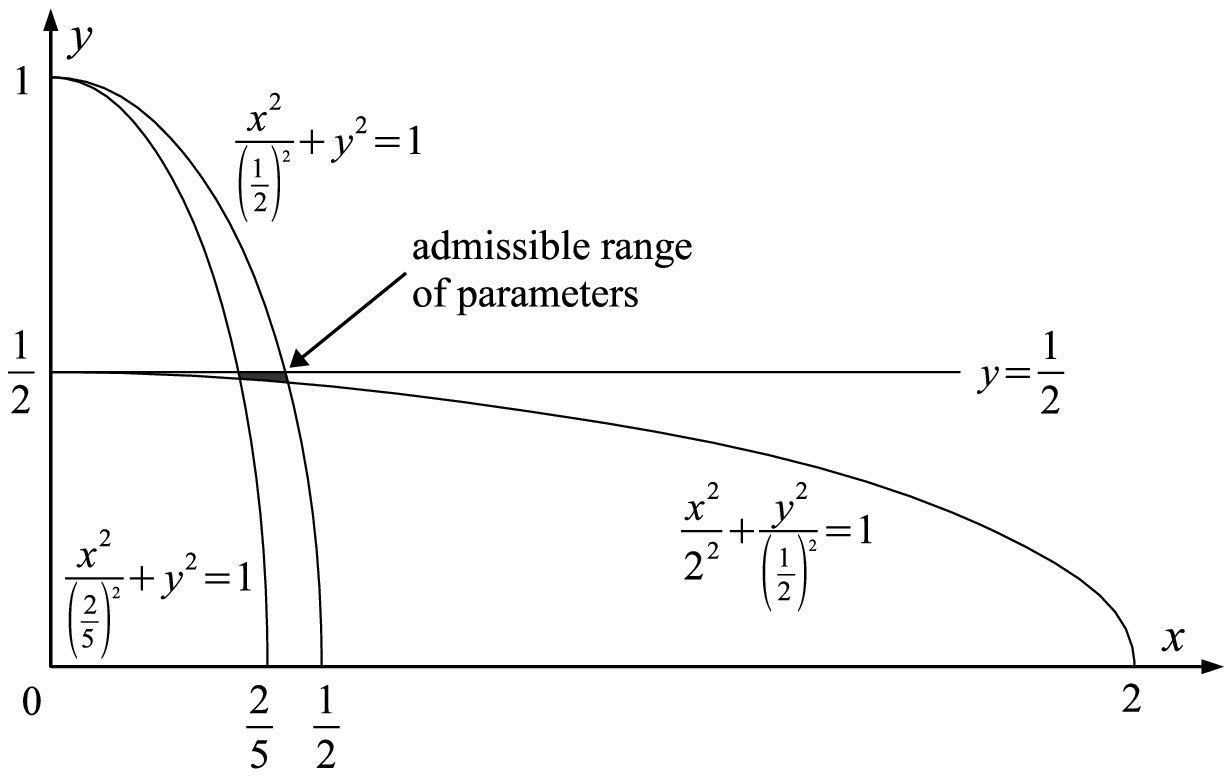}
\end{center}
\end{minipage}\\
(a) & (b)
\end{tabular}
\caption{ (a) The waveguide $\Lambda_{a,b}$.
(b) The range of parameters $x=1/a$ and $y=1/b$ for which the assumptions of Proposition~\ref{propell}
are satisfied. \label{figell}}
\end{figure}

At last we remark that the set of $(a,b)$ given by \eqref{eq-aba}
is non-empty. To see this, denote $x:=\frac{1}{a}$ and $y:=\frac{1}{b}$,
then the conditions \eqref{eq-aba} read as
\[
x>0, \quad 0<y<\dfrac{1}{2}, \quad
\dfrac{x^2}{\big(\frac{1}{2}\big)^2}+y^2<1<\dfrac{x^2}{\big(\frac{2}{5}\big)^2}+y^2,
\quad
1<\dfrac{x^2}{2^2}+\dfrac{y^2}{\big(\frac{1}{2}\big)^2},
\]
and have a simple geometric representation, see Figure~\ref{figell}(b).

\subsection{Three-dimensional configurations}

The analysis of three dimensional domains is much harder due
to a greater variety of possible shapes for both the cross-sections
and the central domains, see e.g.~\cite{bakh,naz12}, so we just
mention two examples.

The first one, $\Lambda_\square$, consists of
three copies of half-infinite cylinders whose cross-section
is a unit square attached to three mutually adjacent faces on a unit cube,
see Figure~\ref{fig-3d}(a). One has $\omega=(0,1)\times(0,1)$
with $\nu=2\pi^2$, and the existence of a non-empty discrete spectrum
follows by the domain monotonicity from the comparison with $\Lambda_{\frac \pi 4}\times(0,1)$, where
$\Lambda_{\frac \pi 4}$ is the broken waveguide of subsection~\ref{ss-brok}.
The associated operator $A_\square:=-\Delta^C_{DN}$ is the Laplacian in $(0,1)^3$
with the Dirichlet-Neumann combination of boundary conditions at each pair of opposite faces,
and its second eigenvalue is $\frac{11\pi^2}{4}>2\pi^2=\nu$. Hence, Corollary~\ref{cor1}
shows  the existence of a unique discrete eigenvalue
and the non-existence of a threshold resonance for $\Lambda_\square$.

The second configuration $\Lambda_{\mathrm{o}}$ consists
of three half-infinite circular cylinders of radius $\frac 1 2$
attached to three mutually adjacent faces of a unit cube,
see Figure~\ref{fig-3d}(b). One has then $\nu=4j_{0,1}^2$ with $j_{0,1}\simeq 2.405$
being the first zero of the Bessel function $J_0$, i.e. $\nu\simeq 23.1$,
and the existence of at least one discrete eigenvalue follows from the comparision
with a sharply bent  infinite cylinder of radius $\frac 1 2$ contained in $\Lambda_\mathrm{o}$,
see~\cite{gold}. The associated operator $-\Delta^C_{DN}$
can be minorated by the respective operator $A_\square$ from the previous example,
hence, $\lambda_2(-\Delta^C_{DN})\ge \frac{11\pi^2}{4}\simeq 27.1>\nu$, and Corollary~\ref{cor1}
shows that $\Lambda_\mathrm{o}$ has a single discrete eigenvalue
and no threshold resonance.

\begin{figure}
\centering
\begin{tabular}{cc}
\begin{minipage}[c]{70mm}
\begin{center}
\includegraphics[height=35mm]{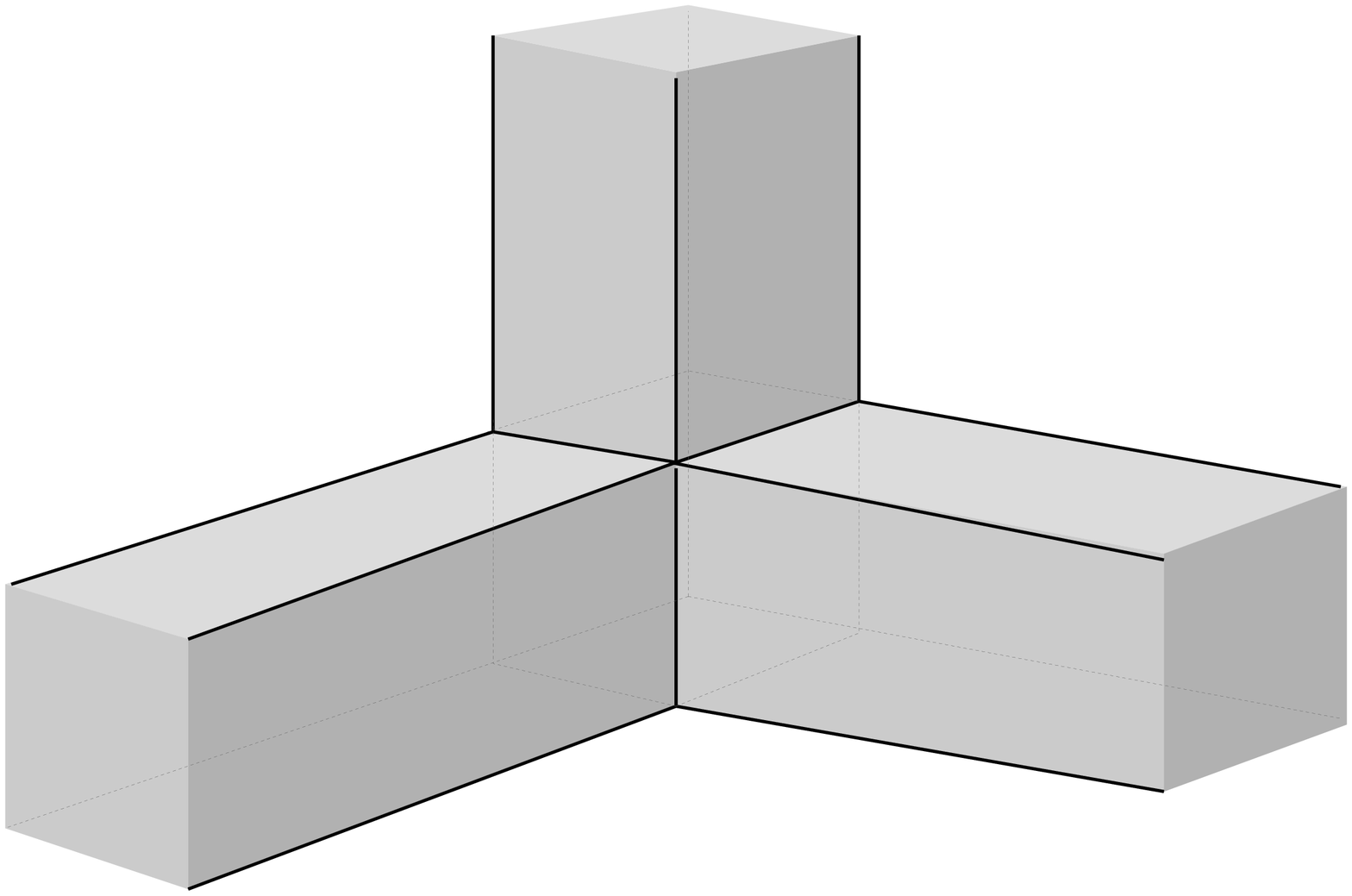}
\end{center}
\end{minipage}
&
\begin{minipage}[c]{70mm}
\begin{center}
\includegraphics[height=35mm]{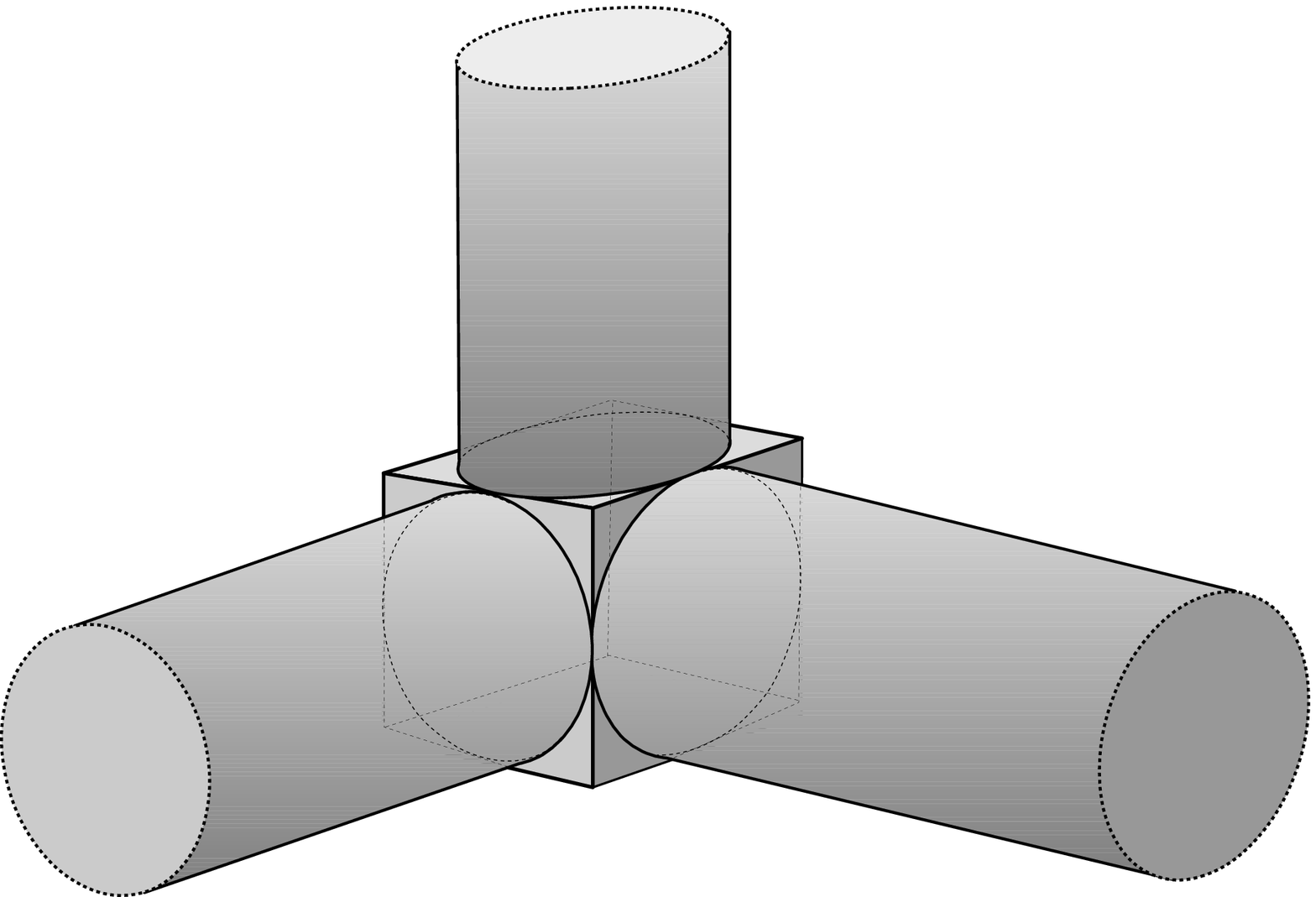}
\end{center}
\end{minipage}\\
(a) & (b)
\end{tabular}
\caption{(a) The three-dimensional waveguide $\Lambda_\square$. (b) The three-dimensional waveguide $\Lambda_\mathrm{o}$.\label{fig-3d}}
\end{figure}

In \cite{bakh,bakh2}, the intersection of two circular cylinders was considered, and
the analysis was more involved. In particular, it was shown using an asymptotic
estimate that the conditions
of Corollary~\ref{cor1} are satisfied if one chooses a sufficiently big center.

\appendix

\section{Proof of Proposition~\ref{prop10}}\label{proof10}
Recall that the sesquilinear form for $-\Delta^{\Lambda_\alpha}_D$ is
$q_{\alpha}(u,v)=\langle \nabla u, \nabla v \rangle_{L^2(\Lambda_\alpha)}$,
$u,v\in H^1_0(\Lambda_\alpha)$.
The domains $\Omega^\pm_\alpha:=\Lambda_\alpha \mathop{\cap} (\RR\times \RR_\pm)$
are isometric to $\Pi_\alpha :=\big\{(s,t): \,t\in(0,1), \, s+t\cot \alpha>0\big\}$
using the representation
\[
\Omega^\pm_\alpha= \big\{
s \sigma^\pm_\alpha+t\tau^\pm_\alpha: (s,t)\in \Pi_\alpha
\big\}, \quad
\sigma^\pm_\alpha:=(\cos\alpha, \pm \sin \alpha),
\quad
\tau^\pm_\alpha:=(-\sin\alpha,\pm \cos\alpha),
\]
see Figure~\ref{fig-st}.
For a function $u$ defined on $\Lambda_\alpha$
we denote by $u^\pm_\alpha$ the functions on $\Pi_\alpha$ defined by
$u^\pm_\alpha(s,t)=u(s\sigma^\pm_\alpha+t \tau^\pm_\alpha)$, $(s,t) \in \Pi_\alpha$,
then
\begin{align}
\|u\|^2_{L^2(\Lambda_\alpha)}&=\sum\nolimits_{\star\in\{+,-\}}
\int_{\Pi_\alpha} u^\star_\alpha (s,t)^2 \dd s\, \dd t, \nonumber\\
q_\alpha(u,u)&=\sum\nolimits_{\star\in\{+,-\}}
\int_{\Pi_\alpha} \bigg[\Big( \dfrac{\partial u^\star_\alpha}{\partial s } (s,t)\Big)^2+
\Big( \dfrac{\partial u^\star_\alpha}{\partial t } (s,t)\Big)^2 \bigg] \dd s\, \dd t.
  \label{eq-qst}
\end{align}

The linear map $\Phi_{\alpha,\beta}:L^2(\Lambda_\alpha)\to L^2(\Lambda_\beta)$
defined by
\[
\big(\Phi_{\alpha,\beta}\big)^\pm_\beta(s,t)= \sqrt{\dfrac{\tan\beta}{\tan\alpha}}u^\pm_\alpha\Big(\dfrac{\tan\beta}{\tan\alpha} s,t\Big),
\quad
(s,t)\in \Pi_\beta,
\]
is unitary with $\Phi_{\alpha,\beta} \big(H^1_0(\Lambda_\alpha)\big)=H^1_0(\Lambda_\beta)$,
and with the help of \eqref{eq-qst} one shows that for any $u,v\in H^1_0(\Lambda_\alpha)$ there holds
\begin{equation}
      \label{eq-qab}
q_\beta\big(\Phi_{\alpha,\beta}u,\Phi_{\alpha,\beta}v\big)
=q_\alpha(u,v)+\bigg(\Big(\dfrac{\tan\beta}{\tan\alpha}\Big)^2-1\bigg)
\sum\nolimits_{\star\in\{+,-\}}\int_{\Omega^\star_\alpha}
(\sigma^\star_\alpha\cdot \nabla u\big)(\sigma^\star_\alpha\cdot \nabla v\big)\, \dd x.
\end{equation}

\begin{figure}
\centering
\begin{tabular}{cc}
\begin{minipage}[c]{65mm}\includegraphics[height=40mm]{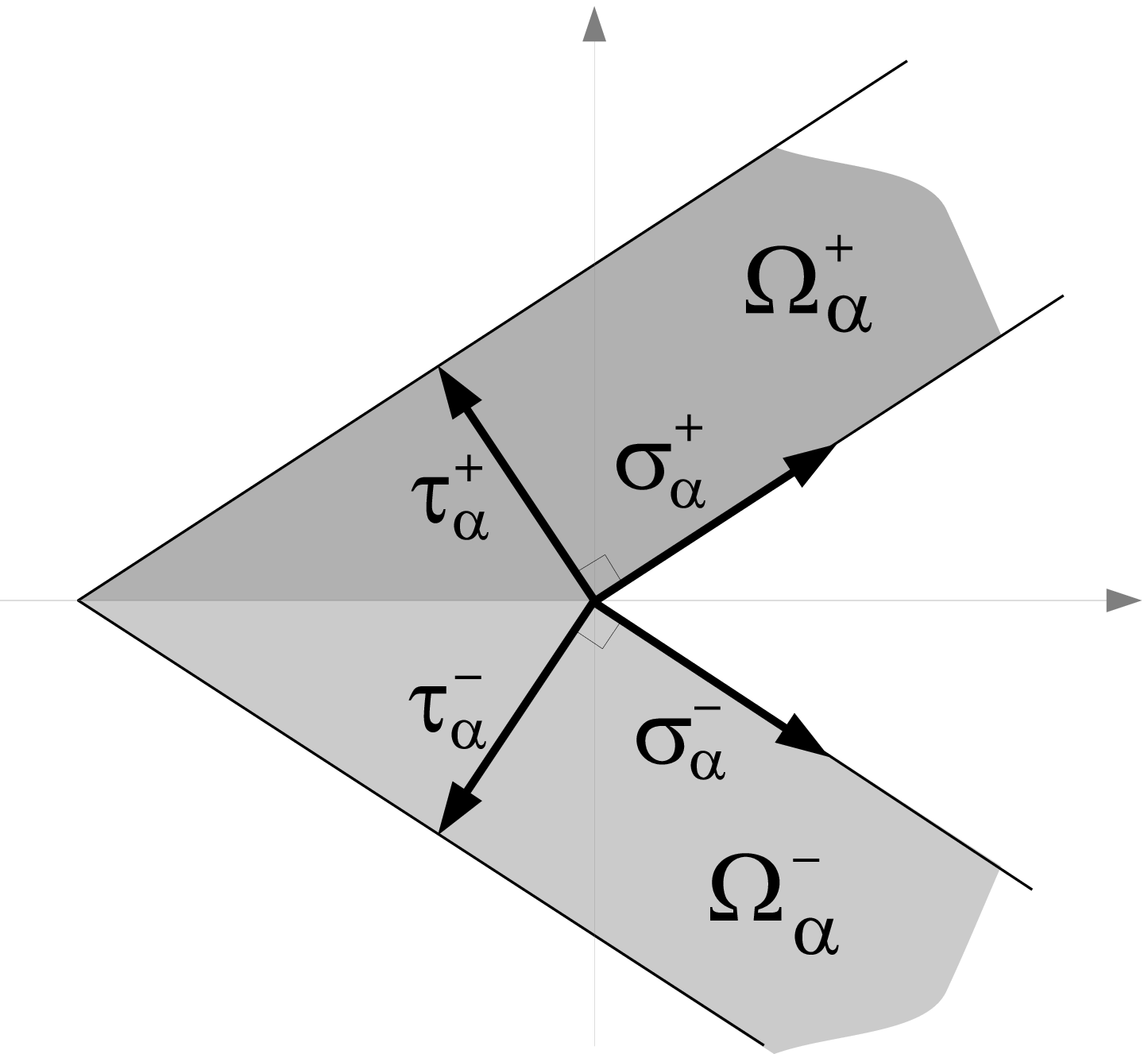}\end{minipage}
&
\begin{minipage}[c]{65mm}\includegraphics[height=25mm]{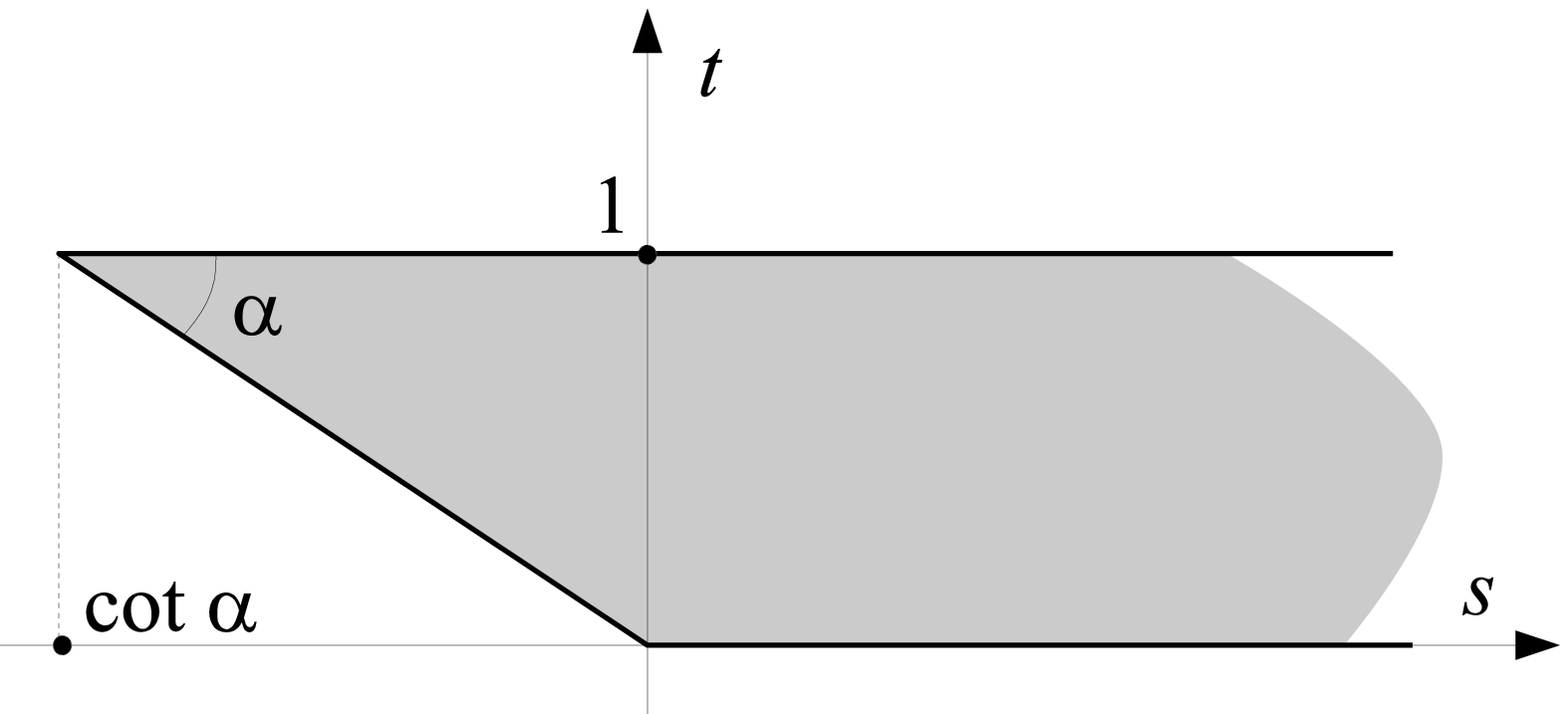}\end{minipage}\\
~\\
(a) & (b)
\end{tabular}

\caption{(a) Decomposition of $\Lambda_\alpha$ from subsection~\ref{ss-brok}.
(b) The set $\Pi_\alpha$.\label{fig-st}}
\end{figure}

Assume that $-\Delta^{\Lambda_\alpha}_D$ has exactly $n$
eigenvalues in $(-\infty,\pi^2)$, to be denoted $\lambda_1,\dots,\lambda_n$,
and choose an associated orthonormal family of eigenfunctions
of $-\Delta^{\Lambda_\alpha}_D$, i.e.
\begin{equation}
   \label{eq-ujuk}
\langle u_j,u_k\rangle_{L^2(\Lambda_\alpha)}=\delta_{jk},
\quad
-\Delta u_j=\lambda_j u_j, \quad j\in\{1,\dots, n\}.
\end{equation}
Furthermore, by assumption there exists a non-zero
bounded solution $u_0$ to \eqref{eq-bd} with $\Lambda=\Lambda_\alpha$.

Let $\beta\in (0,\alpha)$. Denote for shortness
$\gamma:=1-\Big(\dfrac{\tan\beta}{\tan\alpha}\Big)^2>0$.
We will show that $-\Delta^{\Lambda_\beta}_D$
has at least $n+1$ eigenvalues in $(-\infty,\pi^2)$.
By the min-max principle, it is sufficient to show that there exists a
linearly independent family
$(v_0,\dots v_n)\subset H^1_0(\Lambda_\beta)$ such that
\begin{gather}
      \label{eq-mmm}
\sup_{\xi\in \RR^{n+1},\, |\xi|=1} \langle \xi, M\xi\rangle_{\RR^{n+1}}<0,\\
M=(m_{jk}),
\quad
m_{jk}=q_\beta(v_j,v_k)-\pi^2\langle v_j,v_k\rangle_{L^2(\Lambda_\beta)},
\quad j,k\in\{0,\dots,n\big\}.\nonumber
\end{gather}
We construct such a family as follows. Let us pick a $C^\infty$ cut-off function
$\chi:\RR\to[0,1]$ with $\chi(r)=1$ for $r\le 1$
and $\chi(r)=0$ for $r\ge 2$ and define $\varphi:\Lambda_\alpha\to \RR$ by
$\varphi(x)= \chi\big(x/|R|\big)$
with some $R>(\sin \alpha)^{-1}$, to be chosen later (the condition $R>(\sin \alpha)^{-1}$
ensures that support of $\varphi$ covers the ``tip'' on the domain),
and set
$v_0:=\Phi_{\alpha,\beta}(\varphi u_0)$ and $v_j=\Phi_{\alpha,\beta}(u_j)$
for $j\in\{1,\dots,n\}$.
As $\Phi_{\alpha,\beta}$ is an isomorphism, it follows
from Lemma~\ref{lem2} that $v_0,\dots,v_n$ are linearly independent.
Denote $B=(b_{jk})$ with
\[
b_{jk}:= \sum\nolimits_{\star\in\{+,-\}}\int_{\Omega^\star_\alpha}\big(\sigma^\star_\alpha\cdot \nabla v_j\big)\big(\sigma^\star_\alpha\cdot \nabla v_k\big)\, \dd x,
\quad j,k\in\{0,\dots,n\big\},
\]
then due to \eqref{eq-qab} we can represent $M=A-\gamma B$ with $A=(a_{jk})$
with $a_{jk}$ given by \eqref{eq-coefa}, and the estimates of Subsection~\ref{ss21} show that,
with a suitable $a>0$,
\begin{equation}
    \label{eq-aaa2}
\sup_{\xi\in \RR^{n+1},\, |\xi|=1} \langle \xi, A\xi\rangle_{\RR^{n+1}} \le a R^{-\frac 12} \text{ for } R\to+\infty.
\end{equation}
Let us show that
\begin{equation}
     \label{eq-bbb2}
\text{there exists $b>0$ such that } 
\inf_{\xi\in \RR^{n+1},\, |\xi|=1} \langle \xi, B\xi\rangle_{\RR^{n+1}}\ge b \text{ for } R\to +\infty.
\end{equation}
We remark first that for any $\xi=(\xi_0,\dots,\xi_n)\in\RR^{n+1}$ there holds
\begin{equation}
    \label{eq-loc1}
\langle \xi, B\xi\rangle_{\RR^{n+1}}=
\sum\nolimits_{\star\in\{+,-\}}\int_{\Omega^\star_\alpha} \Big(\sigma^\star_\alpha\cdot \nabla (\xi_0\varphi u_0 +\sum\limits_{j=1}^n \xi_j u_j\big)\Big)^2\, \dd x.
\end{equation}
Choose some $R_0>(\sin\alpha)^{-1}$ and denote
$\Omega:=\Lambda_\alpha\mathop{\cap} \big\{ x\in \RR^2: |x|<R_0\big\}$.
As the subintegral function in \eqref{eq-loc1} is non-negative
and $\varphi=1$ on $\Omega$ for $R\ge R_0$, we arrive at
\[
\langle \xi, B\xi\rangle_{\RR^{n+1}}\ge 
\sum\nolimits_{\star\in\{+,-\}}\int_{\Omega\mathop{\cap}\Omega^\star_\alpha} \bigg(\sigma^\star_\alpha\cdot \nabla
\Big(\sum\nolimits_{j=0}^n \xi_j u_j\Big)\bigg)^2\, \dd x=:I(\xi),
\]
and to prove \eqref{eq-bbb} it is sufficient to check that
$\inf_{\xi\in \RR^{n+1},\, |\xi|=1} I(\xi)>0$.
Assume that the inequality is false, then due to the compactness of the unit ball of $\RR^{n+1}$
there exists $\xi=(\xi_0,\dots,\xi_n)$ with $|\xi|=1$ such that $I(\xi)=0$.
As the subintegral expression is non-negative, this implies
\begin{equation}
      \label{eq-ss}
\sigma^\pm_\alpha\cdot \nabla \Big(\sum\nolimits_{j=0}^n \xi_j u_j\Big)=0 \text{ in } \Omega\mathop{\cap} \Omega^\pm_\alpha.
\end{equation}
As each $u_j$ is a (generalized) Laplacian eigenfunction, it is $C^2$ inside $\Lambda_\alpha$,
and, due to \eqref{eq-ss},
\[
\sum\nolimits_{j=0}^n \xi_j u_j (x)= \psi^\pm (\tau^\pm_\alpha\cdot x), \quad x\in \Omega\mathop{\cap} \Omega^\pm_\alpha
\]
with some $C^2$ functions $\psi^\pm: (0,1)\to \RR$. Furthermore, the function $w$ given by
\[
w(x)=\psi^\pm (\tau^\pm_\alpha\cdot x) \text{ for } x\in \Omega\mathop{\cap} \Omega^\pm_\alpha,
\]
coincides with a linear combination of $u_j$ and, hence,
extends to a $C^2$ function in $\Omega$. In particular,
\[
w(x_1,0-)=w(x_1,0+), \quad
\dfrac{\partial w}{\partial x_2}(x_1,0-)=\dfrac{\partial w}{\partial x_2}(x_1,0+),
\quad x_1\in\big(-(\sin\alpha)^{-1}, 0\big),
\]
which results in the the following conditions for $\psi^\pm$, valid
for all $x_1\in\big(-(\sin\alpha)^{-1}, 0\big)$:
\[
\psi^-(-x_1 \sin \alpha)=\psi^+(-x_1 \sin \alpha),\quad
(-\cos \alpha) (\psi^-)'(-x_1 \sin \alpha)= (\cos \alpha) (\psi^+)'(-x_1 \sin \alpha).
\]
The first condition shows that $\psi^+=\psi^-=:\psi$, and the second one implies
that $\psi$ is constant. As the above-mentioned function $w$
satisfies the Dirichlet boundary conditions at
$\partial \Omega \mathop{\cap} \partial \Lambda_\alpha$,
we have $\psi\equiv 0$ and $\xi_0 u_0+\dots+\xi_j u_j=0$ in $\Omega$,
and $\xi=0$ by Lemma~\ref{lem1}. This contradiction with $|\xi|=1$
shows the claim \eqref{eq-bbb2}.
Finally, the combination of \eqref{eq-aaa2} and \eqref{eq-bbb2} shows
that the sought inequality \eqref{eq-mmm} is valid for large $R$.

\end{document}